\documentclass[10pt,reqno]{amsart}

\usepackage{amssymb,amsmath,amstext,amsthm,amsfonts}
\usepackage{amsfonts}
\usepackage{amsmath}
\usepackage{amssymb}
\usepackage{amsthm}
\usepackage{pstricks,pstricks-add,pst-math,pst-xkey}
\newcommand{\C}{\mathbb{C}}
\newcommand{\R}{\mathbb{R}}
\newcommand{\g}{\mathfrak{g}}

\newcommand{\lt}{\mathfrak{t}}

\newcommand{\bdm}{\begin{displaymath}}
\newcommand{\edm}{\end{displaymath}}

\theoremstyle{definition}
\newtheorem{thm}{Theorem}

\newtheorem{lem}[thm]{Lemma}
\newtheorem{defn}{Definition}
\newtheorem{prop}[thm]{Proposition}
\newtheorem{rem}{Remark}
\newtheorem{eg}{Example}

\newtheorem{cor}{Corollary}

\title[]{Harmonic spheres in outer symmetric spaces, their canonical elements and Weierstrass-type representations}

\author{N. Correia and R. Pacheco}
\address{Universidade da Beira Interior\\
Rua Marqu\^{e}s d'\'{A}vila e Bolama, 6200-001 Covilh\~{a}, Portugal}
\email{ncorreia@ubi.pt, rpacheco@ubi.pt}
\begin{document}

\begin{abstract}
 Making use of Murakami's classification of outer involutions in a Lie algebra and following the Morse-theoretic approach to harmonic two-spheres in Lie groups introduced by Burstall and Guest, we
obtain a new classification of harmonic two-spheres in outer symmetric spaces and a  Weierstrass-type representation for such maps. Several examples of harmonic maps into classical outer symmetric spaces are given in terms of meromorphic functions on $S^2$.
\end{abstract}

\maketitle

  \section{Introduction}\label{introd}

  The  harmonicity of  maps $\varphi$ from a Riemann surface $M$ into a compact Lie group $G$ with identity $e$ amounts to the flatness of  one-parameter families of connections. This establishes a correspondence between such maps and certain holomorphic maps $\Phi$ into the based loop group $\Omega G$, the \emph{extended solutions} \cite{Uh}. Evaluating an extended solution $\Phi$ at $\lambda=-1$ we obtain a harmonic map $\varphi$ into the Lie group. If
  an extended solution takes values in  the group of algebraic loops $\Omega_\mathrm{alg}{G}$,
 the corresponding harmonic map is said
to have \textit{finite uniton number}. It is well known that all harmonic maps from the two-sphere into a compact Lie group have finite uniton number \cite{Uh}.

Burstall and Guest \cite{BG} have used a method suggested by Morse theory in order to describe
harmonic maps with finite uniton number from $M$ into a compact Lie group $G$ with trivial centre.  One of the  main ingredients in that paper is the Bruhat decomposition
of the group of algebraic loops $\Omega_\mathrm{alg}{G}$.
  Each piece $U_\xi$ of the Bruhat decomposition  corresponds to an element $\xi$ in the integer lattice $\mathfrak{I}(G)=(2\pi)^{-1} \exp^{-1}(e)\cap \mathfrak{t}$  and can be described as the unstable manifold  of  the energy flow on the K\"{a}hler manifold $\Omega_\mathrm{alg}{G}$. Each extended solution $\Phi:M\to \Omega_\mathrm{alg}{G}$ takes values, off some discrete subset $D$ of $M$, in one of these unstable manifolds $U_\xi$ and can be deformed, under the gradient flow of the energy, to an extended solution with values in some conjugacy class of a Lie group homomorphism $\gamma_\xi:S^1\to G$. A normalization procedure allows us to choose $\xi$  among the  \emph{canonical elements} of $\mathfrak{I}(G)$;
there are precisely $2^n$ canonical elements, where $n=\mathrm{rank}(G)$, and consequently $2^n$ classes of harmonic maps. Burstall and Guest \cite{BG} introduced also a Weierstrass-type representation  for such harmonic maps in terms of meromorphic functions on $M$.
It is possible to define a similar notion of canonical element for
compact Lie groups $G$ with non-trivial centre  \cite{correia_pacheco_4,correia_pacheco_5}. In the present paper, we will not assume any restriction on the centre of $G$.

Given an involution $\sigma$ of $G$, the
 compact symmetric $G$-space $N=G/G^\sigma$, where $G^\sigma$ is the subgroup of $G$ fixed by $\sigma$, can be embedded totally geodesically in $G$ via the corresponding Cartan embedding $\iota_\sigma$. Hence harmonic maps into compact symmetric spaces can be interpreted as special harmonic maps into Lie groups. For inner involutions $\sigma=\mathrm{Ad}(s_0)$, where $s_0\in G$ is the geodesic
 reflection at some base point $x_0\in N$, the composition of the Cartan embedding with left multiplication by $s_0$ gives a totally geodesic embedding of $G/G^\sigma$ in $G$ as a connected component of $\sqrt{e}$. Reciprocally, any connected component of $\sqrt{e}$ is a compact inner symmetric $G$-space.
As shown by Burstall and Guest \cite{BG}, any harmonic map into
 a connected component of $\sqrt{e}$ admits an extended solution $\Phi$ which is invariant under the involution $I(\Phi)(\lambda)=\Phi(-\lambda)\Phi(-1)^{-1}$. Off a discrete set, $\Phi$ takes values in  some unstable manifold $U_\xi$ and  can be deformed, under the gradient flow of the energy, to an extended solution with values in some conjugacy class of a Lie group homomorphism $\gamma_\xi:S^1\to G^\sigma$. An appropriate normalization procedure which preserves both $I$-invariance and the underlying connected component of $\sqrt{e}$ allows us  to choose $\xi$  among the canonical elements of $\mathfrak{I}(G^\sigma)$. As a matter of fact, since $\sigma$ is inner, $\mathrm{rank}(G)=\mathrm{rank}(G^\sigma)$ and we have $\mathfrak{I}(G)=\mathfrak{I}(G^\sigma)$, that is the canonical elements of $\mathfrak{I}(G)$ coincide with the canonical elements of $\mathfrak{I}(G^\sigma)$.
 Consequently, if $G$ has trivial center, we have $2^n$ classes of harmonic maps with finite uniton number into \emph{all} inner symmetric $G$-spaces.

The theory of Burstall and Guest \cite{BG} on harmonic two-spheres in compact inner symmetric $G$-spaces was extended by Eschenburg, Mare and Quast \cite{EMQ} to outer symmetric spaces as follows:  each harmonic map from a two-sphere into an outer symmetric space $G/G^\sigma$, with outer involution $\sigma$, corresponds to an extended solution $\Phi$ which is invariant under a certain involution $T_\sigma$ induced by $\sigma$ on $\Omega G$ (see also \cite{GO}); $\Phi$ takes values in some unstable manifold $U_\xi$, off some discrete set; under the gradient flow of the energy any such invariant extended solution is deformed to an extended solution with values in some conjugacy class of a Lie group homomorphism $\gamma_\xi:S^1\to G^\sigma$;   applying the normalization procedure of extended solutions introduced by Burstall and Guest for Lie groups, $\xi$ can be chosen among the canonical elements of $\mathfrak{I}(G^\sigma)\subsetneq \mathfrak{I}(G)$; if $G$ has trivial centre,
there are precisely $2^k$ canonical homorphisms, where $k=\mathrm{rank}(G^\sigma)<\mathrm{rank}(G)$; hence  there are \emph{at most} $2^k$ classes of harmonic two-spheres in $G/G^\sigma$ if $G$ has trivial centre.

 In the present paper, we will establish a more accurate classification of harmonic maps from a two-sphere into compact outer symmetric spaces. This classification takes  into consideration the following crucial facts concerning extended solutions associated to harmonic maps into outer symmetric spaces:
 although any harmonic map from a two-sphere into an outer symmetric space $G/G^\sigma$ admits a $T_\sigma$-invariant extended solution, not all $T_\sigma$-invariant extended solutions correspond to harmonic maps into  $G/G^\sigma$; on the other hand, the Burstall and Guest's normalization procedure  does not necessarily preserve $T_\sigma$-invariance.

Our
 strategy is the following. The existence of outer involutions of a simple Lie algebra $\mathfrak{g}$ depends on the existence of non-trivial involutions of the Dynkin diagram of $\mathfrak{g}^\C$ \cite{BR,EMQ,He,Mu}. More precisely, if $\varrho$ is a  non-trivial involution of the Dynkin diagram of $\mathfrak{g}^\C$, then it induces  an outer involution $\sigma_\varrho$ of  $\mathfrak{g}^\C$, which we call the \emph{fundamental outer involution},  and, as shown by Murakami \cite{Mu}, all the other outer involutions are, up to conjugation, of the form $\sigma_{\varrho,i}:=\mathrm{Ad}\exp\pi\zeta_i\circ \sigma_\varrho$ where each $\zeta_i$ is a certain element in the integer lattice $\mathfrak{I}(G^{\sigma_\varrho})$. Each connected component of $P^{\sigma_\varrho}=\{g\in G| \,\sigma_\varrho(g)=g^{-1}\}$ is a compact outer symmetric $G$-space associated to some involution $\sigma_\varrho$ or $\sigma_{\varrho,i}$; reciprocally, any outer symmetric space $G/G^\sigma$, with $\sigma$ equal to $\sigma_\varrho$ or $\sigma_{\varrho,i}$, can be totally geodesically embedded in the Lie group $G$ as a connected component of $P^{\sigma_\varrho}$ (see Proposition \ref{concomp}).
As shown in Section \ref{basics}, any harmonic map $\varphi$ into a connected component $N$ of $P^{\sigma_\varrho}$ admits a $T_{\sigma_\varrho}$-invariant extended solution $\Phi$;
 off a discrete set, $\Phi$ takes values in  some unstable manifold $U_\xi$.
In Section \ref{normprocedure} we introduce an appropriate normalization procedure  in order to obtain from $\Phi$ a \emph{normalized} extended solution $\tilde{\Phi}$ with values in some unstable manifold $U_\zeta$ such that: $\zeta$ is a
canonical element of $\mathfrak{I}(G^{\sigma_\varrho})$; $\tilde{\Phi}$ is $T_\tau$-invariant, where $\tau$ is the outer involution given by $\tau=\mathrm{Ad}\exp\pi(\xi-\zeta)\circ \sigma_{\varrho}$;  $\tilde{\Phi}(-1)$ takes values  in some connected component of $P^{\sigma_\varrho}$ which  is an isometric copy of $N$ completely determined by $\zeta$ and $\tau$; moreover, $\tilde{\Phi}(-1)$ coincides with $\varphi$ up to isometry.
 Hence, we obtain a classification of harmonic maps of finite uniton number from $M$ into outer symmetric $G$-spaces in terms of the pairs $(\zeta,\tau)$.

 Dorfmeister, Pedit and Wu \cite{DPW} have introduced a general scheme for constructing harmonic maps from a Riemann surface into a compact symmetric space from holomorphic data, in which the harmonic map equation reduces to a linear ODE similar to the classical Weierstrass representation of minimal surfaces. Burstal and Guest \cite{BG} made this scheme more explicit for the case $M=S^2$ by establishing a ``Weierstrass formula" for harmonic maps with finite uniton number into Lie groups and their inner symmetric spaces.
 In Theorem \ref{sigmaweirstrass} we establish a version of this formula to  outer symmetric  spaces, which  allows us to describe the corresponding $T_\sigma$-invariant extended solutions in terms of meromorphic functions on $M$. For normalized extended solutions and ``low uniton number", such descriptions are easier to obtain. In Section \ref{examples} we give several explicit examples of harmonic maps from the two-sphere into classical outer symmetric spaces:  Theorem \ref{RAs} interprets old results by Calabi \cite{calabi_1967} and Eells and Wood \cite{eells_wood_1983} concerning harmonic spheres in real projective spaces $\mathbb{R}P^{2n-1}$ in view of our classification; harmonic two-spheres into the real Grassmannian $G_3(\mathbb{R}^6)$ are studied in detail; we show that all harmonic two spheres into the \emph{Wu manifold} $SU(3)/SO(3)$ can be obtained explicitly by choosing two meromorphic functions on $S^2$ and then performing a finite number of algebraic operations, in agreement with the explicit constructions established by  H. Ma in \cite{Ma}.

 \section{Groups of algebraic loops}

For completeness, in this section we recall some fundamental facts concerning the structure of the group of algebraic loops in a compact Lie group. Further details can be found in \cite{BG,correia_pacheco_3,PS}.

%\subsection{The Bruhat decomposition}
Let $G$ be a compact matrix semisimple Lie group with Lie algebra $\mathfrak g$  and identity $e$. Denote the \emph{free} and \emph{based} loop groups of $G$ by $\Lambda G$  and $\Omega G$, respectively, whereas $\Lambda_+ G^{\mathbb{C}}$ stands for the subgroup of $\Lambda G^\C$ consisting of loops $\gamma:S^1\to G^\C$ which extend holomorphically to the unitary disc $|\lambda|<1$.

%Taking into consideration the \emph{Iwasawa decomposition}   $\Lambda G^{\mathbb{C}}\cong \Omega G\times \Lambda_+G^\mathbb{C}$,  each $\gamma\in  \Lambda
%G^\mathbb{C}$ can be written uniquely in the form $\gamma=\gamma_R\gamma_+$, with $\gamma_R\in \Omega G$ and $\gamma_+\in \Lambda_+G^\mathbb{C}$. Consequently, there exists a
%\emph{dressing action}  of $\Lambda_+ G$ on
%$\Omega G$: if $g\in\Omega G$ and
%$h\in \Lambda_+ G$, then $h\cdot g=(hg)_R.$

Fix a maximal torus $T$ of $G$ with Lie algebra $\mathfrak{t}\subset \mathfrak{g}$. Let $\Delta\subset \mathrm{i}\mathfrak{t}^*$ be the corresponding
set of roots, where $\mathrm{i}:=\sqrt{-1}$,  and, for each $\alpha\in \Delta$, denote by $\g_\alpha$ the corresponding root space.
Choose a fundamental Weyl  chamber $\mathcal{W}$  in $\mathfrak{t}$, which  corresponds to fix a positive root system $\Delta^+$.  The intersection  $\mathfrak{I}'(G):=\mathfrak{I}(G)\cap\mathcal{W}$ parameterizes the conjugacy classes of homomorphisms $S^1\to G$. More precisely, $\mathrm{Hom}(S^1,G)$ is the
disjoint union of $\Omega_\xi(G)$, with $\xi\in \mathfrak{I}'(G)$,
where   $\Omega_\xi(G)$ is the conjugacy class of homomorphisms  which contains $\gamma_\xi(\lambda)=\exp{(-\mathrm{i}\ln(\lambda)\xi)}$.

 The \emph{Bruhat decomposition } states that the subgroup of algebraic based loops $\Omega_\mathrm{alg}G$ is the disjoint union of the orbits $\Lambda^+_\mathrm{alg}G^\C \cdot\gamma_\xi$, with $\xi\in \mathfrak{I}'(G)$, where $\cdot$ denotes the dressing action of $\Lambda_+G^\mathbb{C}$ on $\Omega G$ induced by the \emph{Iwasawa decomposition}   $\Lambda G^{\mathbb{C}}\cong \Omega G\times \Lambda_+G^\mathbb{C}$. According to the Morse theoretic interpretation \cite{BG,PS} of the Bruhat decomposition, for each $\xi\in\mathfrak{I}'(G)$,
 $U_\xi(G):=\Lambda^+_\mathrm{alg}G^\C\cdot \gamma_\xi$ is  the unstable manifold of $\Omega_\xi(G)$ under the flow induced by  the energy gradient vector field $-\nabla E$: each $\gamma\in U_\xi(G)$ flows to some homomorphism $u_\xi(\gamma)$ in $\Omega_\xi(G)$.

%Each $\xi\in \mathfrak{I}(G)$ endows $\mathfrak{g}^\C$ with a structure of a graded Lie algebra:  for each  $j\in\mathbb{Z}$, let , which is given by the direct sum of those root spaces $\g_\alpha$ satisfying $\alpha(\xi)=j\mathrm{i}$; then
%\begin{equation*}
%\g^\C=\!\!\!\!\!\bigoplus_{j\in\{-r(\xi)\ldots,r(\xi)\}}\!\!\!\!\!\mathfrak{g}^\xi_j,\quad [\mathfrak{g}^\xi_i, \mathfrak{g}^\xi_j]\subset \mathfrak{g}^\xi_{i+j},
%\end{equation*}
%where
%$r(\xi)=\mathrm{max}\{j\,|\,\,\g_j^\xi\neq 0\}$.

In \cite{BG}, the authors proved that, for  each $\xi\in\mathfrak{I}'(G)$, the critical manifold $\Omega_\xi(G)$ is  a complex homogeneous space of $G^\C$ and the unstable manifold $U_\xi(G)$ is a complex homogeneous space of the group $\Lambda^+_{\mathrm{alg}}G^\C$.
 Moreover,  $U_\xi(G)$ carries a structure of holomorphic vector bundle over $\Omega_\xi(G)$ and the bundle map
 $u_\xi:U_\xi(G)\to \Omega_\xi(G)$ is precisely the natural projection given by  $[\gamma]\mapsto [\gamma(0)]$.

Define a partial order $\preceq$ over $\mathfrak{I}'(G)$ as follows: $\xi\preceq \xi'$ if $\mathfrak{p}^{\xi}_i\subset \mathfrak{p}^{\xi'}_i$
for all $i\geq 0$, where $\mathfrak{p}_i^\xi=\sum_{j\leq i}\g_j^\xi$ and $\g_j^\xi$ is the $j\mathrm{i}$-eigenspace of $\mathrm{ad}{\xi}$. As shown in \cite{correia_pacheco_3},
 one can  define a $\Lambda^+_{\mathrm{alg}}G^\C$-invariant fibre bundle morphism
$\mathcal{U}_{\xi,\xi'}:U_\xi(G)\to U_{\xi'}(G)$  by
\begin{equation*}
\mathcal{U}_{\xi,\xi'}(\Psi\cdot \gamma_{\xi})=\Psi\cdot\gamma_{\xi'}, \quad \Psi\in\Lambda^+_{\mathrm{alg}}G^\C,\end{equation*}
whenever  $\xi\preceq \xi'$.
Since the holomorphic structures on  $U_\xi(G)$ and $U_{\xi'}(G)$ are induced by the holomorphic structure on $\Lambda^+_{\mathrm{alg}}G^\C$, the fibre-bundle morphism  $\mathcal{U}_{\xi,\xi'}$ is holomorphic.

\section{Harmonic spheres in Lie groups}

Harmonic maps from the two-sphere $S^2$ into a compact matrix Lie group $G$ can be classified in terms of certain pieces of the Bruhat decomposition of $\Omega_\mathrm{alg}G$. Next we recall briefly this theory from \cite{BG,correia_pacheco_3,correia_pacheco_4,correia_pacheco_5}.

\subsection{Extended Solutions}
Let  $M$ be a simply-connected Riemann surface, $\varphi:M\rightarrow G$ be a smooth map and $\rho:G\to \mathrm{End}(V)$ a finite representation of $G$. Equip $G$ with a bi-invariant metric.
If $\varphi$ is an harmonic map of \emph{finite uniton number},  it admits an extended solution $\Phi:M\to \Omega G$ with $\Phi(M)\subseteq \Omega_{\mathrm{alg}}G$ and  $\varphi=\Phi_{-1}$. In this case, we can write
$\rho\circ\Phi=\sum_{i=r}^s\zeta_i\lambda^i$ for some $r\leq s\in\mathbb{Z}$.  The number $s-r$ is called the \emph{uniton number} of $\Phi$ with respect to $\rho$, and the minimal value of $s-r$ (with respect to all extended solutions associated to $\varphi$) is called the \emph{uniton number} of $\varphi$ with respect to $\rho$ and it is denoted by $r_\rho(\varphi)$.
As explained in \cite{correia_pacheco_5}, this definition of  uniton number of an extended solution with respect to the adjoint representation is twice that of Burstall and Guest \cite{BG}.  K. Uhlenbeck \cite{Uh} proved that all harmonic maps from the two-sphere have finite uniton number. For simplicity of exposition, henceforth we will take $M=S^2$. However, all our results still hold for harmonic maps of finite uniton number from an arbitrary Riemann surface.

%Define $\alpha=\varphi^{-1}{d}\varphi$ and let $\alpha=\alpha'+\alpha''$
%be the type decomposition of $\alpha$ into $(1,0)$ and
%$(0,1)$-forms. As  first observed by K. Uhlenbeck \cite{Uh},
% $\varphi:M\rightarrow G $ is harmonic if and only if $d+\alpha_\lambda$, with
%$\alpha_\lambda=\frac12(1-\lambda^{-1})\alpha'+\frac12(1-\lambda)\alpha''$, is a $S^1$-family of flat connections.
%Then, if  $\varphi$ is harmonic and $M$ is simply connected, we can
%integrate to obtain a map $\Phi:M\rightarrow \Omega G$, the \textit{extended solution} associated to $\varphi$, such that
%$\alpha_\lambda=\Phi_\lambda^{-1}{d}\Phi_\lambda$ and $\Phi_{-1}=\varphi$. If $\tilde\Phi=\gamma\Phi$ for some $\gamma\in \Omega G$, we say that the extended solutions $\tilde \Phi$ and $\Phi$ are \emph{equivalent}.
%
%
%
%
%An

\begin{thm}\cite{BG}\label{usd}
 {Let $\Phi:S^2\to \Omega_{\mathrm{alg}}G$ be an extended solution. Then there exists some $\xi\in \mathfrak{I}'(G)$, and some discrete subset $D$ of $S^2$, such that $\Phi(S^2\setminus D)\subseteq U_\xi(G)$.}
  \end{thm}
%This is a consequence of the holomorphicity of $\Phi$. As a matter of fact, we have the following.
%\begin{thm}\cite{correia_pacheco_3} A smooth map $\Phi:M\setminus D\to U_\xi$ is an extended solution if, and only if, $\Phi$ is holomorphic and pseudo-horizontal
%  (that is, the derivative of $\Phi$ along $1,0$-direction takes values in $H^{1,0}U_\xi$).
%\end{thm}

Given a smooth map $\Phi:S^2\setminus D\to U_\xi(G)$, consider $\Psi:S^2\setminus D \to \Lambda_{\mathrm{alg}}^+G^\C$ such that $\Phi =\Psi\cdot\gamma_\xi $, that is
 $\Psi\gamma_\xi=\Phi b$ for some $b:S^2\setminus D\to \Lambda^+_{\mathrm{alg}}G^\C.$
Write
$\Psi^{-1}\Psi_z=\sum_{i\geq 0} X'_i\lambda^i$, and $\Psi^{-1}\Psi_{\bar{z}}=\sum_{i\geq 0} X''_i\lambda^i.$
Proposition 4.4 in  \cite{BG} establishes that   $\Phi$ is an extended solution if, and only if,
\begin{equation}\label{im}
\mathrm{Im} X'_i\subset \,\mathfrak{p}^\xi_{i+1},\,\,\,\,\mathrm{Im} X''_i\subset \mathfrak{p}^\xi_{i},
\end{equation}
where $\mathfrak{p}_i^\xi=\bigoplus_{j\leq i}\g_j^\xi$ and $\g_j^\xi$ is the $j\mathrm{i}$-eigenspace of $\mathrm{ad}{\xi}$.  The derivative of the harmonic map $\varphi=\Phi_{-1}$ is given by the following formula.
\begin{lem}\label{poi}\cite{correia_pacheco_3}
  Let  $\Phi=\Psi\cdot\gamma_\xi:S^2\to \Omega_{\mathrm{alg}}G$ be an extended solution and  $\varphi=\Phi_{-1}:S^2\to G$ the corresponding harmonic map. Then
  $$\varphi^{-1}\varphi_z=-2\sum_{i\geq 0}b(0){X'_i}^{i+1}b(0)^{-1},$$
  where ${X'_i}^{i+1}$ is the component of ${X'_i}$ over $\g^\xi_{i+1}$, with respect to the decomposition $\g^\C=\bigoplus \g^\xi_j$.
\end{lem}

Both the fiber bundle morphism $\mathcal{U}_{\xi,\xi'}:U_\xi(G)\to U_{\xi'}(G)$ and  the bundle map $u_\xi:U_\xi(G)\to \Omega_\xi(G)$ preserve harmonicity.
\begin{prop}\cite{BG,correia_pacheco_3} \label{popo}
Let $\Phi:S^2\setminus D\to U_\xi(G)$  be an extended solution. Then
\begin{enumerate}
  \item[a)]  $u_\xi\circ\Phi:S^2\setminus D\to \Omega_\xi(G)$ is an extended solution, with $\xi\in\mathfrak{I}(G)$;
  \item[b)] for each $\xi'\in \mathfrak{I}'(G)$ such that $\xi\preceq \xi'$,
$\mathcal{U}_{\xi,\xi'}(\Phi)=\mathcal{U}_{\xi,\xi'}\circ \Phi:S^2\setminus D\to U_{\xi'}(G)$ is an extended solution.
  \end{enumerate}
\end{prop}

\subsection{Weierstrass representation.}\label{weircond}  Taking a larger discrete subset  if necessary, one obtains a more explicit description for harmonic maps of finite uniton number and their extended solutions as follows.

  \begin{prop}\label{BG}\cite{BG}
  Let $\Phi:S^2\to \Omega_{\mathrm{alg}}G$ be an extended solution.
  There exists a discrete set $D'\supseteq D$ of $S^2$  such that
  $\Phi{\big|_{S^2\setminus D'}}=\exp C\cdot \gamma_\xi$ for some  holomorphic vector-valued function $C: S^2\setminus D'\to \mathfrak{u}^0_\xi$,
  where $\mathfrak{u}^0_\xi$ is the finite dimensional nilpotent subalgebra of $\Lambda^+_{\mathrm{alg}}\g^\C$ defined by
$$\mathfrak{u}^0_\xi=\bigoplus_{0\leq i<r(\xi)}\lambda^i(\mathfrak{p}^\xi_i)^\perp,\quad (\mathfrak{p}^\xi_i)^\perp=\bigoplus _{i<j\leq r(\xi)}\mathfrak{g}_j^\xi,$$
with $r(\xi)=\mathrm{max}\{j\,|\,\,\g_j^\xi\neq 0\}$.
Moreover, $C$ can be extended meromorphically to $S^2$.
\end{prop}

Conversely, taking account \eqref{im} and the well-known formula for the derivative of the exponential map, we see that  if $C: S^2\to \mathfrak{u}^0_\xi$ is meromorphic
then $\Phi=\exp{C}\cdot \gamma_{\xi}$ is an
 extended solution if and only if  in the expression
\begin{equation}\label{C_z}
(\exp C)^{-1}(\exp C)_z=C_z-\frac{1}{2!}(\mathrm{ad} C)C_z+\ldots +(-1)^{r(\xi)-1}\frac{1}{r(\xi)!}(\mathrm{ad} C)^{r(\xi)-1}C_z,
  \end{equation}
  the coefficient $\lambda^i$  have zero component in each $\g_{i+2}^\xi,\ldots,\g^\xi_{r(\xi)}$.

\subsection{$S^1$-invariant extended solutions}
Extended solutions with values in some $\Omega_\xi(G)$, off a discrete subset, are said to be \emph{$S^1$-invariant}.
If we take a unitary representation $\rho:G \to U(n)$ for some $n$, then for any such extended solution $\Phi$ we have
$\rho\circ \Phi_\lambda=\sum_{i=r}^s\lambda^i\pi_{W_i},$
where, for each $i$,  $\pi_{W_i}$ is the orthogonal projection onto a complex vector subbundle $W_i$ of $\underline{\C}^n:=M\times \C^n$ and $\underline{\C}^n=\bigoplus_{i=r}^sW_i$ is an orthogonal direct sum decomposition.
Set $A_i=\bigoplus_{j\leq i }W_j$ so that
$$
\{0\}\subset A_r \subset \ldots \subset A_{i-1}\subset A_i\subset A_{i+1}\subset\ldots \subset A_s= \underline{\C}^n.$$
The harmonicity condition amounts to the following conditions on this flag: for each $i$, $A_i$ is a holomorphic subbundle of  $\underline{\C}^n$; the flag  is \emph{superhorizontal}, in the sense that, for each $i$, we have $\partial A_i\subseteq A_{i+1}$, that is, given any section $s$ of $A_i$ then $\frac{\partial s}{\partial z}$ is a section of $A_{i+1}$ for  any local complex coordinate $z$  of $S^2$.

\subsection{Normalization of harmonic maps}\label{BGnorm}

Let  $\Delta_0:=\{\alpha_1,\ldots,\alpha_r\}\subset \Delta^+$ be the basis of positive simple roots, with dual basis $\{H_1,\ldots, H_r\}\subset\mathfrak{t}$, that is $\alpha_i(H_j)=\mathrm{i}\,\delta_{ij}$, where $r=\mathrm{rank}(\mathfrak{g})$.
Given $\xi=\sum n_iH_i$ and $\xi'=\sum n'_iH_i$ in $\mathfrak{I}'({G})$,  we have $n_i,n'_i\geq 0$ and observe that  $\xi\preceq\xi'$ if and only if $n'_i\leq n_i$ for all $i$.
For each $I\subseteq \{1,\ldots,r\}$, define the cone
$\mathfrak{C}_{I}=\Big\{\sum_{i=1}^r n_i H_i|\, n_i\geq 0, \,\mbox{$n_j=0$ iff $j\notin I$}\Big\}.$

\begin{defn}\cite{correia_pacheco_5} Let $\xi\in\mathfrak{I}'({G})\cap \mathfrak{C}_{I}$. We say that $\xi$ is a \emph{$I$-canonical element} of $G$ with respect to  $\mathcal{W}$ if it is a maximal element of $(\mathfrak{I}'({G})\cap \mathfrak{C}_{I},\preceq)$, that is:  if $\xi\preceq \xi'$ and $\xi'\in \mathfrak{I}'({G})\cap \mathfrak{C}_{I}$ then $\xi=\xi'$.\end{defn}

When $G$ has trivial centre, which is the case considered in \cite{BG}, there exists a unique $I$-canonical element, which is given by $\xi_I=\sum_{i\in I}H_i$,  for each $I$. When $G$ has non-trivial centre,  the $I$-canonical elements of $G$ were described in \cite{correia_pacheco_4,correia_pacheco_5}.

Any harmonic map $\varphi:S^2\to G$ admits a \emph{normalized extended solution}, that is, an extended solution $\Phi$ taking  values in $U_\xi(G)$, off some discrete set, for some canonical element $\xi$. This is a consequence of the following  generalization of Theorem 4.5 in \cite{BG}.

\begin{thm}\label{nor}\cite{correia_pacheco_3} Let $\Phi:S^2\setminus D\to U_\xi(G)$ be an extended solution. Take $\xi'\in \mathfrak{I}'({G})$ such that $\xi\preceq {\xi'}$ and $\g_0^\xi=\g_0^{\xi'}$. Then $\gamma^{-1}:=\mathcal{U}_{\xi,\xi-\xi'}(\Phi)$ is a constant loop in $\Omega_{\mathrm{alg}}{G}$ and $\gamma\Phi:S^2\setminus D\to U_{\xi'}(G)$. \end{thm}

The uniton number of a normalized extended solution $\Phi:S^2\setminus D\to U_\xi(G)$  can be computed with respect to any irreducible $n$-dimensional representation ${\rho}:G\to\mathrm{End}(V)$ with highest weight $\omega^*$ and lowest weight  $\varpi^*$ as follows \cite{correia_pacheco_5}: $r_\rho(\xi):=\omega^*(\xi)-\varpi^*(\xi)$.

\section{Harmonic spheres in outer symmetric spaces}
 In the following sections we will establish our classification of  harmonic maps from $S^2$ into compact outer symmetric spaces and establish  a Weierstrass formula for such harmonic maps. These will  allow us to produce some explicit examples of harmonic maps from two-spheres into outer symmetric spaces from meromorphic data.

As we have referred in Section \ref{introd}, although any harmonic map from a two-sphere into an outer symmetric space $G/K$ admits a $T_\sigma$-invariant extended solution, not all $T_\sigma$-invariant extended solutions correspond to harmonic maps into  $G/K$; by Proposition \ref{concomp} and Theorem \ref{GO} below, they correspond to a harmonic map into some possibly different outer symmetric space $G/K'$ (compare Theorem \ref{RAs} with Theorem \ref{36} for an example where this happens).
 %On the other hand, the Burstall and Guest's normalization procedure, as described in Section \ref{BGnorm},  does not necessarily preserve $T_\sigma$-invariance. In Subsection \ref{normprocedure}, we will establish an appropriate normalization procedure for $T_\sigma$-invariant extended solutions.

%The classification of harmonic two-spheres into outer symmetric spaces by Eschenburg, Mare and Quast \cite{EMQ}  does not take into account the following crucial facts concerning extended solutions associated to harmonic maps into outer symmetric spaces: the Burstall and Guest's normalization procedure, as described in Section \ref{BGnorm},  does not necessarily preserve $T_\sigma$-invariance;

% In the present section we will see how the bundle morphisms $\mathcal{U}_{\xi,\xi'}$ act on $T_\sigma$-invariant extended solutions and we will also establish a more convenient normalization procedure of $T_\sigma$-invariant extended solutions.

\subsection{Symmetric $G$-spaces and Cartan embeddings}
Let $N=G/K$ be a symmetric space, where $K$ is the isotropy subgroup at the base point $x_0\in N$, and let $\sigma:G\to G$ be the corresponding involution: we have $G^{\sigma}_0\subseteq K\subseteq G^{\sigma},$ where $G^{\sigma}$ is the subgroup fixed by $\sigma$ and $G_0^{\sigma}$ denotes its connected component of the identity. We assume that $N$ is  a \emph{bottom space}, i.e.  $K=G^{\sigma}$. Let $\mathfrak{g}=\mathfrak{k}_\sigma\oplus \mathfrak{m}_\sigma$ be the $\pm1$-eigenspace decomposition associated to the involution $\sigma$, where $\mathfrak{k}_\sigma$ is the Lie algebra of $K$.
Consider the  (totally geodesic) \emph{Cartan embedding} $\iota_\sigma:N\hookrightarrow G$ defined by $\iota_\sigma (g\cdot x_0)=g\sigma(g^{-1})$. The image of the Cartan embedding is precisely the connected component $P^\sigma_e$ of $P^{\sigma}:=\{g\in G|\,\sigma(g)=g^{-1}\}$ containing the identity $e$ of the group $G$.
Observe that,  given $\xi\in \mathfrak{I}(G)\cap \mathfrak{k}_\sigma$, then $\exp(\pi\xi)\in P^\sigma$.
We denote by
$P_{\xi}^{\sigma}$ the connected component  of
$P^{\sigma}$ containing $\exp(\pi\xi)$.

\begin{prop}\label{concomp}
  Given $\xi\in \mathfrak{I}(G)\cap \mathfrak{k}_\sigma$, we have the following.
\begin{enumerate}
  \item[a)]  $G$ acts transitively on $P_{\xi}^\sigma$ as follows: for $g\in G$ and $h\in  P_{\xi}^{\sigma}$,
\begin{equation}\label{gaction}
g\cdot_\sigma h=gh\sigma(g^{-1}).\end{equation}
\item[b)] $P^{\sigma}_{\xi}$ is a bottom symmetric $G$-space totally geodesically embedded in $G$ with involution
\begin{equation}\label{tau}
\tau=\mathrm{Ad}(\exp\pi\xi)\circ \sigma.
\end{equation}
%In particular, $P^\sigma_{g_0}\cong N$  if and only if $P^\sigma_{g_0}$ contains some element of $Z(G)$.
\item[c)] For any other $\xi'\in \mathfrak{I}(G)\cap\mathfrak{k}_\sigma$   we have $\exp(\pi\xi')\in P^{\tau}$ and
$P^\tau_{\xi'}=\exp(\pi\xi) P_{\xi'-\xi}^\sigma.$
\item[d)]
The $\pm 1$-eigenspace decomposition $\mathfrak{g}=\mathfrak{k}_\tau\oplus\mathfrak{m}_\tau$ associated to the symmetric $G$-space $P_{\xi}^\sigma$ at the fixed point $\exp(\pi\xi)\in P_{\xi}^\sigma$ is given by
\begin{align}
\label{hc} \mathfrak{k}_\tau^\C&= \bigoplus \mathfrak{g}^\xi_{2i}\cap \mathfrak{k}_\sigma^\C\oplus \bigoplus\mathfrak{g}^\xi_{2i+1}\cap \mathfrak{m}_\sigma^\C\\  \label{pc} \mathfrak{m}_\tau^\C&=\bigoplus \mathfrak{g}^\xi_{2i+1}\cap \mathfrak{k}_\sigma^\C\oplus \bigoplus\mathfrak{g}^\xi_{2i}\cap \mathfrak{m}_\sigma^\C.
\end{align}
\end{enumerate}
\end{prop}

\begin{proof}

Take $h\in P^\sigma$. We have
$$\sigma(g\cdot_\sigma h)=\sigma(gh\sigma(g^{-1}))=\sigma(g)h^{-1}g^{-1}=(gh\sigma(g^{-1}))^{-1}=(g\cdot_\sigma h)^{-1}.$$
Then $g\cdot_\sigma h\in P^\sigma$ and we have a continuous action of $G$ on $P^\sigma$. Since $G$ is connected, this action induces an action of $G$ on each connected component  of $P^\sigma$.
Since $g\cdot_\sigma e=g\sigma(g^{-1})=\iota_\sigma(g\cdot x_0)$ and $\iota_\sigma(N)=P_e^\sigma$, the action $\cdot_\sigma$  of $G$ on
$P_e^\sigma$ is transitive.

Take $\xi\in \mathfrak{I}(G)\cap \mathfrak{k}_\sigma$, so that $\sigma(\xi)=\xi$ and $\exp2\pi\xi=e$. Consider the involution $\tau$ defined by \eqref{tau}. If $g\in P^\sigma$, then
 $$\tau(\exp(\pi\xi)g)=\exp(\pi\xi)\sigma(\exp(\pi\xi)g)\exp(\pi\xi)=\sigma(g)\exp(\pi\xi)=(\exp(\pi\xi)g)^{-1},$$  which means that $\exp(\pi\xi)g\in P^\tau$. Reciprocally, if $\exp(\pi\xi)g\in P^\tau$, one can check similarly that $g\in P^{\sigma}$. Hence
   $P^\tau=\exp(\pi\xi)P^\sigma$. In particular, by continuity, $P^\tau_{\xi'}=\exp(\pi\xi) P^\sigma_{\xi'-\xi}$  for any other $\xi'\in \mathfrak{I}(G)$ with $\sigma(\xi')=\xi'$.

 Reversing the rules of $\sigma= \mathrm{Ad}(\exp\pi\xi)\circ \tau$ and $\tau$,  we also have $P^\sigma_{\xi}=\exp(\pi\xi)P_e^\tau.$
 Since $G$ acts transitively on $P_e^\tau$, for each $h\in P^\sigma_{\xi}$ there exists $g\in G$ such that
 $$h=\exp(\pi\xi) (g\cdot_\tau e)=(\exp(\pi\xi)g)\cdot_\sigma \exp(\pi\xi).$$ This shows that $G$ also acts transitively on
$P^\sigma_{\xi}$. The isotropy subgroup at $\exp(\pi\xi)$ consists of those elements $g$ of $G$ satisfying $g\exp(\pi\xi)\sigma(g^{-1})=\exp(\pi\xi)$,
that is  those elements $g$ of $G$ which are fixed by $\tau$:
\begin{equation}\label{isot}
\exp(\pi\xi)\sigma(g)\exp(\pi\xi)=g.
\end{equation}
Hence
$P^\sigma_{\xi}\cong G/G^{\tau},$ which is a bottom symmetric $G$-space with involution $\tau$.
Since $P_e^\tau\subset G$ totally geodesically and $P^\sigma_\xi$ is the image of $P_e^\tau$ under an isometry (left multiplication by $\exp\pi\xi$), then  $P_{\xi}^\sigma\subset G$  totally geodesically.

% If $P^\sigma_{g_0}\cong N$, then $\tau=Ad_g\circ \sigma \circ Ad_{g^{-1}}$ for some $g\in G$. Hence $Ad_{g_0}=Ad_{g\sigma(g^{-1})}$, which means that $g_0=zg\sigma(g^{-1})=g\cdot_\sigma z$ for some $z\in Z(G)$, and consequently $z\in P_{g_0}^\sigma$.

Differentiating \eqref{isot} at the identity we get
$\mathfrak{k}_\tau=\{X\in \mathfrak{g}|\, X=\mathrm{Ad}(\exp\pi\xi)\circ\sigma(X)\}.$
Taking account of the  formula $\mathrm{Ad}({\exp (\pi\xi)})=e^{\pi \mathrm{ad} \xi}$ and that $\sigma$ commutes with $\mathrm{ad}\xi$,
we obtain  \eqref{hc}; and \eqref{pc} follows similarly.
\end{proof}

\subsubsection{Outer symmetric spaces.}
%If $s_0$ is the geodesic reflection at  $x_0\cong eK$ and $G$ is the group $I(N)$ of isometries of $N$, the involution $\sigma$ is given by $\sigma=\mathrm{Ad}(s_0)$.  Recall that the symmetric space $N=G/K$ is \emph{inner} if $s_0\in I_0(N)$ and \emph{outer} otherwise, where $I_0(N)$ is the connected component of $I(N)$ containing the identity. It is well known that  $N$ is inner if and only if $\mathrm{rank}(K)=\mathrm{rank}(G)$ \cite{He}.

The existence of outer involutions of a simple Lie algebra $\mathfrak{g}$ depends on the existence of non-trivial involutions of the Dynkin diagram of $\mathfrak{g}^\C$ \cite{BR,EMQ,He,Mu}.
Fix a maximal abelian subalgebra  $\mathfrak{t}$ of $\mathfrak{g}$ and a Weyl chamber $\mathcal{W}$ in $\mathfrak{t}$, which amounts to fix a system of positive simple roots $\Delta_0=\{\alpha_1,\ldots,\alpha_r\}$, where $r=\mathrm{rank}(\mathfrak{g})$. Let $\varrho$  be a non-trivial involution of the Dynkin diagram and $\sigma_\varrho$  the \emph{fundamental outer involution} associated to $\varrho$  \cite{BR,Mu}.  The (local isometry classes of) outer symmetric spaces of compact type associated to involutions of the form $\sigma_\varrho$ are precisely
\begin{quote}$SU(2n)/Sp(n)$, $SU(2n+1)/SO(2n+1)$, $E_6/F_4$ and  the real projective spaces  $\mathbb{R}P^{2n-1}$.\end{quote}
These spaces are called the \emph{fundamental outer symmetric spaces}. The remaining classes of outer involutions are obtained as follows \cite{BG,Mu}.
%Extend $\varrho$ by linearity and duality to give an involution of $\mathfrak{t}$. This is the restriction of $ \sigma_\varrho$ to $\mathfrak{t}$.  For a  suitable choice of root %vectors $X_\alpha$ of $\mathfrak{\g}_\alpha$, with $\alpha \in \Delta_0$, the restriction of $ \sigma_\varrho$ to the span of these vectors is given by $ %\sigma_\varrho(X_\alpha)=X_{\varrho(\alpha)}$. The \emph{fundamental outer involution} $\sigma_\varrho$ associated to $\varrho$ is the unique extension of this to an outer involution %of $\mathfrak{g}$.

Let $\mathfrak{g}=\mathfrak{k}_{\varrho}\oplus \mathfrak{m}_{\varrho}$ be the corresponding  $\pm 1$-eigenspace decomposition of $\mathfrak{g}$.
As shown in Proposition 3.20 of \cite{BR}, the Lie subalgebra $\mathfrak{k}_{\varrho}$ is simple and the orthogonal projection of $\Delta_0$ onto $\mathfrak{k}_{\varrho}$,  $\pi_{\mathfrak{k}_{\varrho}}(\Delta_0)$,  is a basis of positive simple roots of $\mathfrak{k}_\varrho$ associated to the maximal abelian subalgebra $\mathfrak{t}_{\mathfrak{k}_{\varrho}}:=\mathfrak{t}\cap \mathfrak{k}_{\varrho}$.
Consider the split $\mathfrak{t}=\mathfrak{t}_{{\mathfrak{k}_\varrho}}\oplus \mathfrak{t}_{\mathfrak{m}_\varrho}$ with respect to  $\mathfrak{g}={\mathfrak{k}_\varrho}\oplus {\mathfrak{m}_\varrho}$.
Set $s=r-k$, where $k=\mathrm{rank}(\mathfrak{k}_\varrho)$. We can label the basis $\Delta_0$ in order to  get the following relations: $\varrho (\alpha_j)=\alpha_j$ for $1\leq j\leq k-s$ and  $\varrho (\alpha_j)=\alpha_{s+j}$ for $k-s+1\leq j\leq k$. Let $\pi_{\mathfrak{k}_\varrho}$ be the orthogonal projection of $\mathfrak{t}$ onto $\mathfrak{t}_{{\mathfrak{k}_\varrho}}$, that is
$\pi_{\mathfrak{k}_\varrho}(H)=\frac12(H+\sigma_\varrho (H))$ for all $H\in \mathfrak{t}$. Set $\pi_{{\mathfrak{k}_\varrho}}(\Delta_0)=\{\beta_1,\ldots,\beta_k\}$, with
\begin{equation}\label{bes}
\beta_j=\left\{\begin{array}{cl} \alpha_j & \mbox{for $1\leq j\leq k-s$}
            \\ \frac12(\alpha_j+\alpha_{j+s}) & \mbox{for $k-s+1\leq j\leq k$}
\end{array}\right..\end{equation}
This  is a basis of $\mathrm{i}\mathfrak{t}_{\mathfrak{k}_\varrho}^*$ with dual basis $\{\zeta_1,\ldots, \zeta_k\}$ given by
\begin{equation}\label{zes}
\zeta_j=\left\{\begin{array}{cl}
            H_j & \mbox{for $1\leq j\leq k-s$}  \\
           H_j+H_{j+s} & \mbox{for $k-s+1\leq j\leq k$}
          \end{array}\right..\end{equation}
          \begin{thm} \cite{Mu}\label{murak}
 Let $\varrho$ be an involution of the Dynkin diagram of $\mathfrak{g}$. Let
 $$\omega=\sum_{j=1}^{k-s}n_j\beta_j+\sum_{j=k-s+1}^{k}n'_j\beta_j$$  be the highest root of $\mathfrak{k}_\varrho$ with respect to $\pi_{\mathfrak{k}_{\varrho}}(\Delta_0)=\{\beta_1,\ldots,\beta_k\}$, defined as in \eqref{bes}. Given $i$ such that  $n_i=1$ or $2$, define an involution $\sigma_{\varrho,i}$ by
  \begin{equation}\label{sigmasis}
  \sigma_{\varrho,i}=\mathrm{Ad}(\exp\pi \zeta_i)\circ \sigma_\varrho.\end{equation} Then any outer involution of $\mathfrak{g}$ is conjugate in $\mathfrak{Aut}(\mathfrak{g})$, the group of automorphisms of $\mathfrak{g}$, to some $\sigma_\varrho$ or  $\sigma_{\varrho,i}$. In particular, there are at most $k-s+1$ conjugacy classes of outer involutions.
\end{thm}
The list of all (local isometry classes of) irreducible outer symmetric spaces of compact type is shown in Table 1 (cf. \cite{BR,EMQ, He}).
\begin{table}[!htb]
\begin{tabular}{|c |c| c| c|c|}
  \hline
  \small{$G/K$} & $\mathrm{rank}(G)$ & $\mathrm{rank}(K)$ & $\mathrm{rank}(G/K)$ & $\mathrm{dim}(G/K)$ \\  \hline
  $SU(2n)/SO(2n)$ & $2n-1$ & $n$ & $2n-1$ & $(2n-1)(n+1)$\\ \hline
  $SU(2n+1)/SO(2n+1)$ & $2n$ & $n$ & $2n$ & $n(2n+3)$ \\ \hline
  $SU(2n)/Sp(n)$  & $2n-1$ & $n$ & $ n-1$ &  $(n-1)(2n+1)$ \\ \hline
  $G_{p}(\mathbb{R}^{2n})$ ($p$ odd $\leq n$) & $n$ & $n-1$ & $p$ & $p(2n-p)$\\ \hline
  $E_6/Sp(4)$ & $6$ & $4$ & $6$ & $42$ \\ \hline
  $E_6/F_4$ & $6$ & $4$ & $2$& $26$ \\
  \hline
\end{tabular}
 \vspace{.03in}

\caption{Irreducible outer symmetric spaces.}
 \end{table}

Given an outer involution $\sigma$ of the form $\sigma_{\varrho,i}$ or $\sigma_\varrho$ and its $\pm 1$-eigenspace decomposition $\g=\mathfrak{k}_\sigma\oplus \mathfrak{m}_\sigma$,
set $\mathfrak{t}_{\mathfrak{k}_\sigma}=\mathfrak{t}\cap \mathfrak{k}_\sigma$, which is a maximal abelian subalgebra of  $\mathfrak{k}_\sigma$. Following \cite{EMQ}, a non-empty intersection of  $\mathfrak{t}_{\mathfrak{k}_\sigma}$ with a Weyl chamber in $\mathfrak{t}$ is called a \emph{compartment}.  Each compartment lies in a Weyl chamber in  $\mathfrak{t}_{\mathfrak{k}_\sigma}$ and the Weyl chambers in $\mathfrak{t}_{\mathfrak{k}_\sigma}$ can be decomposed into the same number of compartments \cite{EMQ}.

The intersection of the integer lattice  $\mathfrak{I}(G)$ with the Weyl chamber $\mathcal{W}$ in $\mathfrak{t}$, which we have denoted by $\mathfrak{I}'(G)$, is described in terms of the dual basis $\{H_1,\ldots,H_r\}\subset \mathfrak{t}$, with $r=\mathrm{rank} (\mathfrak{g})$, by
$$ \mathfrak{I}'(G)=\big\{\sum_{i=1}^r n_iH_i\in \mathfrak{I}(G) |\,\mbox{$n_i\in \mathbb{N}_0$ for all $i$}\big\}.$$
When $\sigma$ is a  fundamental outer involution $\sigma_\varrho$, the compartment $\mathcal{W}\cap \mathfrak{t}_{\mathfrak{k_\varrho}}$ is itself a Weyl chamber in $\mathfrak{t}_{\mathfrak{k_\varrho}}$. Then, the intersection of  the integer lattice  $\mathfrak{I}(G^{\sigma_\varrho})$ with the Weyl chamber $\mathcal{W}\cap \mathfrak{t}_{\mathfrak{k_\varrho}}$,
is given by
$$ \mathfrak{I}'(G^{\sigma_\varrho})=\big\{\sum_{i=1}^k n_i\zeta_i\in \mathfrak{I}(G) |\,\mbox{$n_i\in \mathbb{N}_0$ for all $i$}\big\}=\mathfrak{I}'(G)\cap \mathfrak{t}_{\mathfrak{k_\varrho}}.$$

\subsubsection{Cartan embeddings of fundamental outer symmetric spaces.}

Next we describe those elements $\xi$ of $ \mathfrak{I}'(G^{\sigma_\varrho})$ for which the connected component $P_\xi^{\sigma_\varrho}$ of $P^{\sigma_\varrho}$ containing $\exp(\pi\xi)$ can be identified with the fundamental outer symmetric $G$-space associated to $\varrho$.
Start by considering the following $\sigma_\varrho$-invariant subsets  of the  root system $\Delta\subset \mathrm{i}\mathfrak{t}^*$ of $\mathfrak{g}$:
\begin{align}\label{deltas}
 \Delta(\mathfrak{k}_\varrho)=\{\alpha\in\Delta|\, \g_\alpha\subset \mathfrak{k}_\varrho^\C\},\,\,
   \Delta(\mathfrak{m}_\varrho)=\{\alpha\in\Delta|\, \g_\alpha\subset \mathfrak{m}_\varrho^\C\},\,\, \Delta_\varrho=\Delta\setminus  \left(\Delta(\mathfrak{k}_\varrho)\cup  \Delta(\mathfrak{m}_\varrho)\right).
\end{align}
 Then
\begin{equation*}
   \mathfrak{k}_\varrho^\C=\mathfrak{t}^\C_{\mathfrak{k}_\varrho}\oplus\pi_{\mathfrak{k}_\varrho}(\mathfrak{r}_\varrho)\oplus\bigoplus_{\alpha\in \Delta(\mathfrak{k}_\varrho)}\g_\alpha,\,\,\,  \mathfrak{m}_\varrho^\C=\mathfrak{t}^\C_{\mathfrak{m}_\varrho}\oplus\pi_{\mathfrak{m}_\varrho}(\mathfrak{r}_\varrho)\oplus\bigoplus_{\alpha\in \Delta(\mathfrak{m}_\varrho)}\g_\alpha,
\end{equation*}
where $\mathfrak{r}_\varrho=\bigoplus_{\alpha\in\Delta_\varrho}\g_\alpha$. Since the involution $\varrho$ acts on $\Delta_\varrho$ as a permutation without fixed points, we can fix some subset $\Delta'_\varrho$ so that $\Delta_\varrho$ is the disjoint union of $\Delta'_\varrho$ with $\varrho(\Delta'_\varrho)$:
\begin{equation}\label{delta1}
  \Delta_\varrho= \Delta'_\varrho {\sqcup}\,\, \varrho(\Delta'_\varrho).
\end{equation}
For each $\alpha \in\Delta'_\varrho$,   $\sigma_\varrho$ restricts to an involution in the subspace $\g_\alpha\oplus \g_{\varrho(\alpha)}\subset \mathfrak{r}_\varrho$. Hence we have the following.
\begin{lem}\label{pmmm}The orthogonal projections of $\mathfrak{r}_\varrho$ onto $\mathfrak{k}_\varrho^\C$ and $\mathfrak{m}_\varrho^\C$ are given by
 $$\pi_{\mathfrak{k}_\varrho}(\mathfrak{r}_\varrho)=\!\bigoplus_{\alpha\in\Delta'_\varrho}\!  \mathfrak{k}_\varrho^\C\cap \big(\g_\alpha\oplus\g_{\varrho(\alpha)}\big),\,\,\pi_{\mathfrak{m}_\varrho}(\mathfrak{r}_\varrho)=\!\!\bigoplus_{\alpha\in\Delta'_\varrho}\!  \mathfrak{m}_\varrho^\C\cap \big(\g_\alpha\oplus\g_{\varrho(\alpha)}\big),$$
 and, for each  $\alpha\in\Delta'_\varrho$,
 $$ \mathfrak{k}_\varrho^\C\cap \big(\g_\alpha\oplus\g_{\sigma_\varrho(\alpha)}\big)=\{X_\alpha+\sigma_\varrho(X_\alpha)|\, X_\alpha\in \g_\alpha\},\,\,\, \mathfrak{m}_\varrho^\C\cap \big(\g_\alpha\oplus\g_{\sigma(\alpha)}\big)=\{X_\alpha-\sigma_\varrho(X_\alpha)|\, X_\alpha\in \g_\alpha\}.$$
 In particular, $\dim \mathfrak{r}_\varrho=2\dim \pi_{\mathfrak{k}_\varrho}(\mathfrak{r}_\varrho)= 2\dim  \pi_{\mathfrak{m}_\varrho}(\mathfrak{r}_\varrho)$.
\end{lem}

\begin{prop}\label{fundcan}
Consider the  dual basis $\{\zeta_1,\ldots,\zeta_k\}$ defined by \eqref{zes}.  Given $\xi \in \mathfrak{I}'(G^{\sigma_\varrho})$ with $\xi=\sum_{i=1}^kn_i\zeta_i$ and $n_i\geq 0$, then $P_\xi^{\sigma_\varrho}$ is a fundamental outer symmetric space with involution (conjugated to) $\sigma_\varrho$ if and only if $n_i$ is even for each $1\leq i\leq k-s$.
\end{prop}
\begin{proof}
There is only one class of outer symmetric $SU(2n+1)$-spaces and, in this case, the involution $\varrho$ does not fix any simple root, that is $k-s=0$. Hence the result trivially holds for $N=SU(2n+1)/SO(2n+1).$

Next we consider the remaining fundamental outer symmetric spaces, which are precisely the
symmetric spaces of \emph{rank-split type}  \cite{EMQ}, those satisfying $\Delta(\mathfrak{m}_\varrho)=\emptyset$.
For such symmetric spaces, the reductive symmetric term $\mathfrak{m}_\varrho$ satisfies  $\mathfrak{m}_\varrho=\mathfrak{t}_{\mathfrak{m}_\varrho}\oplus\pi_{\mathfrak{m}_\varrho}(\mathfrak{r}_\varrho)$. On the other hand, in view of \eqref{pc},  we have, for  $\tau=\mathrm{Ad}(\exp\pi\xi)\circ\sigma_\varrho$,
  \begin{align*}\nonumber
    \mathfrak{m}_\tau^\C&=\bigoplus \mathfrak{g}^\xi_{2i+1}\cap \mathfrak{k}^\C_\varrho\oplus \bigoplus\mathfrak{g}^\xi_{2i}\cap \mathfrak{m}^\C_\varrho\\&=
    \mathfrak{t}^\C_{\mathfrak{m}_\varrho}\oplus\!\!\! \bigoplus_{{\alpha\in \Delta(\mathfrak{k}_\varrho)\cap \Delta_\xi^-}}\!\!\! \g_\alpha\oplus
   \!\!\!  \bigoplus_{{\alpha\in \Delta'_\varrho\cap \Delta_\xi^-}} \!\!\!\mathfrak{k}^\C_\varrho\cap(\g_\alpha\oplus  \g_{\varrho(\alpha)})\oplus \!\!\!\bigoplus_{{\alpha\in \Delta'_\varrho\cap\Delta_\xi^+}} \!\!\!\mathfrak{m}^\C_\varrho\cap(\g_\alpha\oplus  \g_{\varrho(\alpha)}),
  \end{align*}
  where $\Delta_\xi^+:=\{\alpha\in\Delta|\,\mbox{$\alpha(\xi)\mathrm{i}$ is even}\}$ and  $\Delta_\xi^-:=\{\alpha\in\Delta|\,\mbox{$\alpha(\xi)\mathrm{i}$ is odd}\}$.
 Taking into account Lemma \ref{pmmm}, from this we see that $\dim \mathfrak{m}_\tau=\dim \mathfrak{m}_\varrho$ (which means, by Table 1, that $P_\xi^{\sigma_\varrho}$ is a fundamental outer symmetric space with involution conjugated to $\sigma_\varrho$) if and only if
 \begin{equation*}
 \bigoplus_{{\alpha\in \Delta(\mathfrak{k}_\varrho)\cap \Delta_\xi^-}}\!\! \g_\alpha=\{0\},\end{equation*}
  which holds if and only if $\xi=\sum_{i=1}^kn_i\zeta_i$ with $n_i$  even for each $1\leq i\leq k-s$.
  \end{proof}

\subsection{Harmonic spheres in symmetric $G$-spaces}\label{basics}

Given an involution $\sigma$ on $G$, define an involution $T_\sigma$ on $\Omega G$ by $T_\sigma(\gamma)(\lambda)=\sigma(\gamma(-\lambda)\gamma(-1)^{-1}).$ Let $\Omega^\sigma G$ be the fixed set of $T_\sigma$.
%\begin{lem}\label{base}
%  If $\gamma\in\Omega^\sigma G$, then $\gamma(-1)\in P^\sigma$.
%\end{lem}
%\begin{proof}
%  If the based loop $\gamma$ is $T_\sigma$-invariant, then $\sigma(\gamma(-\lambda)\gamma(-1)^{-1})=\gamma(\lambda)$, and evaluating at $\lambda=-1$ we get $\sigma(\gamma(-1)^{-1})=\gamma(-1)$, that is $\gamma(-1)\in P^{\sigma}$.
%\end{proof}

\begin{thm}\label{GO}\cite{EMQ,GO}
 Given $\xi\in \mathfrak{I}(G)\cap\mathfrak{k}_\sigma$, any harmonic map $\varphi:S^2\to P^{\sigma}_\xi\subset G$ admits an $T_\sigma$-invariant extended solution $\Phi:S^2\to\Omega^\sigma G$.   Conversely, given an $T_\sigma$-invariant extended solution  $\Phi$, the smooth map $\varphi=\Phi_{-1}$ from $S^2$ is harmonic and takes values  in some connected component of $P^{\sigma}$.
\end{thm}
%\begin{proof}
%Let $\tilde{\Phi}:S^2\to \Omega_{\mathrm{alg}} G$  be an extended solution associated to $\varphi:S^2\to P^{\sigma}_\xi\subset G$, that is $\tilde{\Phi}_{-1}=\varphi$. We assume that for a fixed point $p\in S^2$ we have $\varphi(p)=\gamma_\xi(-1)$. Set $\gamma=\gamma_\xi\tilde\Phi(p)^{-1}$ and $\Phi=\gamma\tilde{\Phi}$. Observe that $\Phi$ is the unique algebraic extended solution satisfying
%$\Phi_{-1}=\varphi$ and $\Phi(p)=\gamma_\xi$.
%A simple computation shows that $T_\sigma(\Phi)$ is also an extended solution associated to $\varphi$ and satisfies  $T_\sigma(\Phi)(p)=\gamma_\xi$. Hence, by unicity, we conclude that
%$\Phi=T_\sigma(\Phi)$. Conversely, if $\Phi$ is $T_\sigma$-invariant, by Lemma \ref{base}, $\Phi_{-1}$ takes values in some connected component of $P^{\sigma}$.
%\end{proof}

%\begin{rem}
%When $N=G/K$ is an \emph{inner} symmetric space and $\sigma=\mathrm{Ad}(s_0)$, with $s_0\in G$ satisfying $s_0^2=e$,  one easily check that $s_0P^{\sigma}\subseteq \sqrt{e}$ and we can identify $N$ with the connected component of $\sqrt{e}=\{h\in G:\,h^2=e\}$ containing $s_0$. Under this identification, harmonic maps into $N$ correspond to extended solutions which are invariant with respect to the involution ${I}:\Omega G\to \Omega G$ given by ${I}(\gamma)(\lambda)=\gamma(-\lambda)\gamma(-1)^{-1}$. This is the point of view used in \cite{BG}.
%\end{rem}

\begin{prop}\cite{EMQ}
\label{incariance} Given $\Phi\in U_\xi^\sigma(G):=U_\xi(G)\cap \Omega^{\sigma}G$, with $\xi\in \mathfrak{I}(G)\cap \mathfrak{k}_\sigma$, set $\gamma=u_\xi\circ \Phi$.
Then  $\gamma$ takes values in $K$. By continuity, $\Phi_{-1}$ and $\gamma(-1)$ take values in the same connected component of $P^\sigma$.
%This connected component is
%$P^\sigma_\xi=\{g\exp(\pi\xi)\sigma(g^{-1})|\,g\in G\}.$
\end{prop}
%\begin{proof}
%  Since the energy $E$ is a $T_\sigma$-invariant function on $\Omega_{\mathrm{alg}}G$, the flow $-\nabla E$ preserves $\Omega^{\sigma}G$. Then, if $\Phi\in U_\xi^{\sigma}(G)$, the loop $\gamma:=u_\xi\circ \Phi\in\Omega_\xi(G)$  is also  $T_\sigma$-invariant, that is $T_\sigma(\gamma) = \gamma$.  A simple computation shows that $\gamma$ takes values in $K$ (see  proof of Lemma 5 in \cite{EMQ}). Again, by continuity  $\Phi_{-1}$ and $\gamma(-1)$ take values in the same connected component of $P^\sigma$.
%\end{proof}

Hence, together with Theorems \ref{usd} and \ref{GO}, this implies the following.

 \begin{thm}\label{tinva}
Any harmonic map $\varphi$ from $S^2$ into a connected component of $P^{\sigma}$  admits an extended solution $\Phi:S^2\setminus D\to U_\xi^{\sigma}(G):=U_\xi(G)\cap \Omega^{\sigma}G$, for some $\xi\in \mathfrak{I}'(G)\cap\mathfrak{k}_\sigma$ and some discrete subset $D$. If $\sigma=\sigma_\varrho$ is the fundamental outer involution, then $\varphi=\Phi_{-1}$ takes values in $P_\xi^{\sigma_\varrho}$.
\end{thm}
\begin{proof}
By Proposition \ref{incariance},  $\Phi$ and $\gamma:=u_\xi\circ \Phi$ take values in the same connected component of $P^{\sigma}$ when evaluated at $\lambda=-1$. Since $\gamma:S^1\to G^\sigma$ is a homomorphism,  $\gamma$ is in the $G^\sigma$-conjugacy class of $\gamma_{\xi'}$ for some $\xi'\in \mathfrak{I}'(G^\sigma)$, where $G^\sigma$ is the subgroup of $G$ fixed by $\sigma$. Consequently, $\gamma(-1)=g\gamma_{\xi'}(-1)g^{-1}=g\cdot_\sigma \gamma_{\xi'}(-1),$
for some $g\in G^{\sigma}$, which means that $\gamma(-1)$
takes values in the connected component $P_{\xi'}^{\sigma}$.
On the other hand,  $\gamma$ is in the $G$-conjugacy class of $\gamma_\xi$, with $\xi\in \mathfrak{I}'(G)\cap\mathfrak{k}_\sigma$. If $\sigma$ is the fundamental outer involution $\sigma_\varrho$, then  $\mathfrak{I}'(G^{\sigma})=\mathfrak{I}'(G)\cap\mathfrak{k}_\sigma$; and we must have $\xi=\xi'$.
\end{proof}
\begin{rem}
  If $\sigma$ is not a fundamental outer involution, each Weyl chamber $\mathcal{W}_\sigma$ in $\mathfrak{t}_{\mathfrak{k}_\sigma}$ can be decomposed into more than one compartment:
  $ \mathcal{W}_\sigma=C_1\sqcup \ldots \sqcup C_l$, where $C_1=\mathcal{W}\cap \mathfrak{t}_{\mathfrak{k}_\sigma}$ and the remaining compartments are conjugate to $C_1$ under $G$ \cite{EMQ}, that is, there exists $g_i\in G$ satisfying $C_i=\mathrm{Ad}(g_i)(C_1)$ for each $i$. Hence, if we have an extended solution $\Phi:S^2\setminus D\to U_\xi^{\sigma}(G)$ with  $\xi\in \mathfrak{I}'(G)\cap\mathfrak{k}_\sigma\subset C_1$, the corresponding harmonic map $\Phi_{-1}$ takes values in one of the connected components $P_{g_i\xi g_i^{-1}}^\sigma$.
  \end{rem}

\subsubsection{$\varrho$-canonical elements.}
Let $I$ be a subset of $\{1,\ldots,k\}$, with $k=\mathrm{rank}(\mathfrak{k}_\varrho)$, and set
$$\mathfrak{C}^\varrho_{I}=\Big\{\sum_{i=1}^k n_i \zeta_i|\, n_i\geq 0, \,\mbox{$n_j=0$ iff $j\notin I$}\Big\}.$$
Let $\xi\in\mathfrak{I}'({G}^{\sigma_\varrho})\cap \mathfrak{C}^\varrho_{I}$. We say that $\zeta$ is a \emph{$\varrho$-canonical element} of $G$ (with respect to the choice of $\mathcal{W}$)  if $\zeta$ is a maximal element of $(\mathfrak{I}'({G}^{\sigma_\varrho})\cap \mathfrak{C}^\varrho_{I},\preceq)$, that is:  if $\zeta\preceq \zeta'$ and $\zeta'\in \mathfrak{I}'({G}^{\sigma_\varrho})\cap \mathfrak{C}^\varrho_{I}$ then $\zeta=\zeta'$.

\begin{rem} When $G$ has trivial centre, the duals $\zeta_1,\ldots,\zeta_k$ belong to the integer lattice. Then, for each $I$ there exists a unique $\varrho$-canonical element, which is given by $\zeta_I=\sum_{i\in I}\zeta_i$. In this case, our definition of $\varrho$-canonical element coincides with that of $S$-canonical element in \cite{EMQ}.
\end{rem}

Now, consider a fundamental outer  involution $\sigma_\varrho$ and let $N$ be an associated outer symmetric $G$-space, that is, $N$ corresponds to an involution of $G$ of the form $\sigma_\varrho$ or $\sigma_{\varrho,i}$, with $\zeta_i$ in the conditions of Theorem \ref{murak}.
If $G$ has trivial centre, we certainly have  $\zeta_i\in \mathfrak{I}'({G}^{\sigma_\varrho})$. As a matter of fact, as we will see later, in most cases we have $\zeta_i\in \mathfrak{I}'({G}^{\sigma_\varrho})$, whether $G$ has trivial centre or not, with essentially one exception: for $G=SU(2n)$ and $N=SU(2n)/SO(2n)$. So, we will treat this case separately and assume henceforth that $\zeta_i\in \mathfrak{I}'({G}^{\sigma_\varrho})$.

\begin{rem}
  Consider the Dynkin diagram of $\mathfrak{e}_6$:
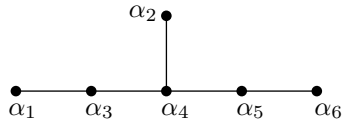
\begin{center}
\psset{xunit=1.0cm,yunit=1.0cm,algebraic=true,dotstyle=o,dotsize=3pt 0,linewidth=0.5pt,arrowsize=3pt 2,arrowinset=0.25}
\begin{pspicture*}(-0.55,-0.5)(4.82,1.30)
\psline(0,0)(1,0)
\psline(1,0)(2,0)
\psline(2,0)(3,0)
\psline(3,0)(4,0)
\psline(2,0)(2,1)
\rput[tl](-0.1,-0.16){$\alpha_1$}
\rput[tl](1.5,1.14){$\alpha_2$}
\rput[tl](0.92,-0.16){$\alpha_3$}
\rput[tl](1.93,-0.16){$\alpha_4$}
\rput[tl](2.93,-0.16){$ \alpha_5 $}
\rput[tl](3.97,-0.16){$\alpha_6$}
\begin{scriptsize}
\psdots[dotsize=4pt 0,dotstyle=*](0,0)
\psdots[dotsize=4pt 0,dotstyle=*](1,0)
\psdots[dotsize=4pt 0,dotstyle=*](2,0)
\psdots[dotsize=4pt 0,dotstyle=*](3,0)
\psdots[dotsize=4pt 0,dotstyle=*](4,0)
\psdots[dotsize=4pt 0,dotstyle=*](2,1)
\end{scriptsize}
\end{pspicture*}
\end{center}
This admits a unique nontrivial involution $\varrho$. Let $\{H_1,\ldots,H_6\}$ be the dual basis of $\Delta_0=\{\alpha_1,\ldots,\alpha_6\}$. The semi-fundamental basis $\pi_{\mathfrak{k}_\varrho}(\Delta_0)=\{\beta_1,\beta_2,\beta_3,\beta_4\}$ is given by
$\beta_1=\alpha_2$, $\beta_2=\alpha_4$, $\beta_3=\frac{\alpha_1+\alpha_6}{2}$ and $\beta_4=\frac{\alpha_3+\alpha_5}{2}$, whereas the dual basis is given by $\zeta_1=H_2$, $\zeta_2=H_4$, $\zeta_3=H_1+H_6$ and $\zeta_4=H_3+H_5$. Taking account that   the elements $H_i$ are related with the duals $\eta_i$ of the fundamental weights  by
$$\left[H_i\right]=\left[\begin{array}{cccccc} 4/3 & 1 & 5/3 & 2 & 4/3 & 2/3\\ 1 & 2 & 2 & 3 & 2 & 1\\ 5/3 & 2 & 10/3 & 4 & 8/3 & 4/3\\ 2 & 3 & 4 & 6 & 4 & 2\\ 4/3 & 2 & 8/3 & 4 & 10/3 & 5/3\\ 2/3 & 1 & 4/3 & 2 & 5/3 & 4/3 \end{array}\right]\left[\eta_i\right],$$
we see that the elements $\zeta_i$ are in the integer lattice $\mathfrak{I}'(\tilde E_6)\subset \mathfrak{I}'( E_6)$, where $\tilde E_6$ is the compact simply connected Lie group with Lie algebra $\mathfrak{e}_6$, which has centre $\mathbb{Z}_3$, and $E_6$ is the adjoint group $\tilde E_6/\mathbb{Z}_3$.
\end{rem}

Taking into account Proposition \ref{concomp}, we can identify $N$ with the connected component $P^{\sigma_\varrho}_{\zeta_i}=\exp(\pi\zeta_i)P_e^{\sigma_{\varrho,i}}$, which is a totally geodesic submanifold of $G$, via
\begin{align}
\label{tttt}
 g\cdot x_0\in N\mapsto\exp(\pi\zeta_i)g\sigma_{\varrho,i}(g^{-1})\in P^{\sigma_\varrho}_{\zeta_i}.
\end{align}
By Theorem \ref{tinva}, each harmonic map $\varphi:S^2\to N\cong P^{{\sigma_\varrho}}_{\zeta_i}$ admits a $T_{\sigma_\varrho}$-invariant extended solution with values, off a discrete set, in some unstable manifold $U_\xi(G)$, with $\xi\in \mathfrak{I}'({G}^{\sigma_\varrho})\cap \mathfrak{C}_I^\varrho$. By Theorem \ref{nor}, this extended solution can be multiplied on the left by a constant loop in order to get a normalized extended solution with values in some unstable manifold $U_\zeta(G)$ for some ${\varrho}$-canonical element $\zeta$. Hence,  if $G$ has trivial centre, the Bruhat decomposition of $\Omega_{\mathrm{alg}}G$ gives rise to  $2^k$ classes of harmonic maps into $P^{\sigma_\varrho}$, that is  $2^k$ classes of harmonic maps into \emph{all} outer symmetric $G$-spaces.

However, the normalization procedure given by Theorem \ref{nor} does not preserve $T_{\sigma_\varrho}$-invariance, and consequently, as we will see next, normalized extended solutions with values in the same unstable manifold $U_\zeta(G)$, for some $\varrho$-canonical element $\zeta$,
correspond in general to harmonic maps into different outer symmetric $G$-spaces. Hence the classification of harmonic two-spheres into outer symmetric $G$-spaces in terms of $\varrho$-canonical elements is manifestly unsatisfactory since it does not distinguishes the underlying symmetric space.
 In the following sections we overcome this weakness by establishing a classification of all such harmonic maps in terms of pairs $(\zeta,\sigma)$, where $\zeta$ is a ${\varrho}$-canonical element and $\sigma$ an outer involution of $G$.

\subsubsection{Normalization of $T_\sigma$-invariant extended solutions}\label{normprocedure}

Let $\sigma$ be an outer involution of $G$. The  fibre bundle morphisms $\mathcal{U}_{\xi,\xi'}$ preserve $T_{\sigma}$-invariance:
\begin{prop}\label{proposition}
 If $\xi\preceq \xi'$ and $\xi,\xi'\in  \mathfrak{I}'({G})\cap \mathfrak{k}_\sigma$, then $\mathcal{U}_{\xi,\xi'}(U_\xi^{\sigma}(G))\subset U_{\xi'}^{\sigma}(G)$.
\end{prop}
\begin{proof} For $\Phi\in U^{\sigma}_\xi(G)$, write $\Phi=\Psi\cdot \gamma_\xi$ for some $\Psi\in \Lambda^+_{\mathrm{alg}}G^\C$.
If $\Phi$ is $T_\sigma$-invariant we have
$\Psi(\lambda)\cdot \gamma_\xi=\sigma(\Psi(-\lambda))\cdot \gamma_\xi.$ Consequently, we also have $\Psi(\lambda)\cdot \gamma_{\xi'}=\sigma(\Psi(-\lambda))\cdot \gamma_{\xi'},$
which means in turn that  $\mathcal{U}_{\xi,\xi'}(\Phi)=\Psi\cdot \gamma_\xi'$ is $T_{\sigma}$-invariant.
\end{proof}

Hence, if $\Phi:S^2\setminus D\to U^{\sigma}_\xi(G)$ is an extended solution and $\xi\preceq\xi'$, with $\xi,\xi'\in\mathfrak{I}'(G)\cap \mathfrak{k}_{_\varrho}$, by Theorem \ref{nor} and Proposition \ref{proposition} we know that $\gamma^{-1}:=\mathcal{U}_{\xi,\xi-\xi'}(\Phi)$ is a constant $T_{\sigma}$-invariant loop if $\g_0^\xi=\g_0^{\xi'}$. However, in general, the product $\gamma \Phi$ is not $T_{\sigma}$-invariant.

\begin{lem}\label{poca}
  Assume that $\gamma^{-1},\Phi\in \Omega^{\sigma} G$ and $\gamma(-1)\in P_{\xi}^{\sigma}$ for some $\xi\in\mathfrak{I}(G)\cap \mathfrak{k}_\sigma$.
  Take $h\in G$ such that $\gamma(-1)=h^{-1}\cdot_\sigma \exp(\pi\xi)$.
  Then $h\gamma\Phi h^{-1}\in\Omega^{\tau} G$, with $\tau=\mathrm{Ad}(\exp \pi\xi)\circ \sigma$.
  \end{lem}
\begin{proof}
Since $\gamma^{-1},\Phi\in \Omega^\sigma G$,   a simple computation shows that $T_{\sigma}(\gamma \Phi)=\gamma(-1)^{-1}\gamma\Phi\gamma(-1).$ Since $\gamma(-1)\in P_{\xi}^{\sigma}$, there exists $h\in G$ such that
$\gamma(-1)=h^{-1}\cdot_\sigma \exp(\pi\xi)=h^{-1}\exp(\pi\xi)\sigma(h)$.  One can check now that $T_\tau(h\gamma \Phi h^{-1})=h\gamma \Phi h^{-1}$.
\end{proof}

\begin{prop}\label{norminf}
   Take
$\xi,\xi '\in \mathfrak{I}'({G})\cap \mathfrak{k}_\sigma$ such that $\xi\preceq \xi'$.
Let  $\Phi:S^2\setminus D\to U^{\sigma}_\xi(G)$ be a $T_{\sigma}$-invariant extended solution.
If $\gamma^{-1}:=\mathcal{U}_{\xi,\xi-\xi'}(\Phi)$ is a constant loop,  there exists $h\in G$ such that
$\tilde{\Phi}:=h\gamma\Phi h^{-1}$ takes values in $U_{\xi'}^\tau(G)$,
with
$\tau=\mathrm{Ad}(\exp\pi(\xi-\xi'))\circ {\sigma}.$

Additionally, if $\sigma$ is the fundamental outer involution $\sigma_\varrho$,
the harmonic map $\Phi_{-1}$ takes values in $P_\xi^{\sigma}$ and  $\tilde{\Phi}_{-1}$ takes values in $P_{\xi'}^{\tau}$, which implies that $\Phi_{-1}$ is given, up to isometry, by
$$\exp(\pi(\xi-\xi'))\tilde{\Phi}_{-1}:S^2\to P_\xi^\sigma.$$
\end{prop}
\begin{proof}
 Assume that $\gamma^{-1}:=\mathcal{U}_{\xi,\xi-\xi'}(\Phi)=\Psi\cdot \gamma_{\xi-\xi'}$ is a constant loop.
We can write $\Psi\gamma_{\xi-\xi'}=\gamma^{-1} b$ for some $b:S^2\setminus D\to \Lambda^+_{\mathrm{alg}}G$. Then
  $$\Phi =\Psi\cdot\gamma_\xi =\Psi\cdot\gamma_{\xi-\xi'}\gamma_{\xi'}= \gamma^{-1} b\cdot\gamma_{\xi'},$$
  which implies that $\gamma\Phi$ takes values  in $U_{\xi'}(G)$. On the other hand, since $\gamma^{-1}$ is $T_\sigma$-invariant (by Proposition \ref{proposition}), $\gamma(-1)\in P^{\sigma}$.

   Take $\eta\in \mathfrak{I}'(G^\sigma)$ and $h\in G$ such that $\gamma(-1)\in P^{\sigma}_{\eta}$ and $\gamma(-1)=h^{-1}\cdot_\sigma\exp\pi\eta$. From  Lemma \ref{poca}, we see that $\tilde\Phi:= h\gamma\Phi h^{-1}$ is $T_\tau$-invariant. Hence $\tilde{\Phi}$ takes values in $U_{\xi'}^\tau(G)$.
 Since $\gamma$ is constant, $\tilde\Phi$ is an extended solution.

If  $\sigma=\sigma_\varrho$, then $\mathfrak{I}'(G^{\sigma_\varrho})=\mathfrak{I}'(G)\cap \mathfrak{k}_{\sigma_\varrho}$, which implies that $\eta=\xi-\xi'$. The element $h\in G$ is such that $$\gamma(-1)=h^{-1}\exp(\pi(\xi-\xi'))\sigma_\varrho(h).$$ On the other hand, since, by Theorem \ref{tinva}, $\Phi_{-1}$
 takes values in $P_\xi^{\sigma_\varrho}$, we also have $\Phi_{-1}=g\exp(\pi\xi)\sigma_\varrho(g^{-1})$ for some lift $g:S^2\to G$. Hence
 \begin{align*}
 \tilde\Phi_{-1}&=h\gamma(-1)\Phi_{-1} h^{-1}=\exp(\pi(\xi-\xi'))\sigma_\varrho(h)g\exp(\pi\xi)\sigma_\varrho(\sigma_\varrho(h)g)^{-1}\\&=\exp(\pi(\xi-\xi'))(\sigma_\varrho(h)g\cdot_{\sigma_\varrho} \exp\pi\xi)
 \end{align*}
Hence, in view of Proposition \ref{concomp}, $\tilde\Phi_{-1}$ takes values in $P_{\xi'}^\tau=\exp(\pi(\xi-\xi'))P_{\xi}^\sigma$.
\end{proof}

Under some conditions on $\xi\preceq \xi'$, the morphism $\mathcal{U}_{\xi,\xi-\xi'}(\Phi)$ is always a constant loop.
\begin{prop}\label{norm2}
 Take
$\xi,\xi '\in \mathfrak{I}'({G})\cap \mathfrak{k}_\sigma$ such that $\xi\preceq \xi'$.
Assume that
 \begin{align}
 \g_{2i}^\xi\cap \mathfrak{m}_\sigma^\C  \subset \bigoplus_{0\leq j<2i}\g^{\xi-\xi'}_j,\quad\quad
 \g_{2i-1}^\xi\cap \mathfrak{k}_\sigma^\C  \subset \bigoplus_{0\leq j<2i-1}\g^{\xi-\xi'}_j,\label{toing2}
 \end{align}
  for all $i> 0$.
 Then, $\mathcal{U}_{\xi,\xi-\xi'}:U_\xi^\sigma(G)\to U_{\xi-\xi'}^\sigma(G)$ transforms $T_\sigma$-invariant extended solutions in constant loops.
 \end{prop}
\begin{proof}
Given an extended solution $\Phi:S^2\setminus D\to U^\sigma_\xi(G)$, choose $\Psi:S^2\setminus D \to \Lambda_{\mathrm{alg}}^+G^\C$ such that $\Phi=\Psi\cdot\gamma_\xi $ and $T_\sigma(\Psi)=\Psi$.
Differentiating this we see that
\begin{equation}\label{ondepsi}
\mathrm{Im}\Psi^{-1}\Psi_z\subset \bigoplus_{i\geq 0}\lambda^{2i}\mathfrak{k}_\sigma^\C\oplus\bigoplus_{i\geq 0}\lambda^{2i+1}\mathfrak{m}_\sigma^\C.
\end{equation}
Write $\Psi^{-1}\Psi_z=\sum_{r\geq 0}\lambda^rX'_r$.
Since $\xi\preceq \xi-\xi'$, by Proposition \ref{popo} and Proposition \ref{proposition}, $\mathcal{U}_{\xi,\xi-\xi'}(\Phi)$ is an extended solution with values in $U^\sigma_{\xi-\xi'}(G)$. Hence, taking into account Lemma \ref{poi},  in order to prove that $\mathcal{U}_{\xi,\xi-\xi'}(\Phi)$ is constant we only have to check that the component of $X'_r$ over $\g^{\xi-\xi'}_{r+1}$ vanishes for all $r\geq 0$.

From \eqref{im} and \eqref{ondepsi} we see that, for   $r=2i$,  $X'_{2i}$ takes values in $\bigoplus_{j\leq 2i+1}  \g_{j}^\xi\cap \mathfrak{k}_\sigma^\C $. But, since $\xi\preceq \xi-\xi'$ and, by hypothesis, \eqref{toing2} holds, we have
$$\bigoplus_{j\leq 2i+1}  \g_{j}^\xi\cap \mathfrak{k}_\sigma^\C =\big(\bigoplus_{j\leq 2i}  \g_{j}^\xi\cap \mathfrak{k}_\sigma^\C\big)\oplus \big(\g_{2i+1}^\xi\cap \mathfrak{k}_\sigma^\C\big)\subset
\big(\bigoplus_{j\leq 2i}  \g_{j}^{\xi-\xi'}\cap \mathfrak{k}_\sigma^\C\big)\oplus  \bigoplus_{0\leq j<2i+1}\g^{\xi-\xi'}_j.$$
 Hence the component of $X'_{2i}$ over $\g^{\xi-\xi'}_{2i+1}$ vanishes for all $i\geq 0$. Similarly,  for   $r=2i-1$,  $X'_{2i-1}$ takes values in $\bigoplus_{j\leq 2i}  \g_{j}^\xi\cap\mathfrak{m}_\sigma^\C $, and we can check that the component of $X'_{2i-1}$ over $\g^{\xi-\xi'}_{2i}$ vanishes for all $i>0$.

 Hence $\gamma^{-1}:=\mathcal{U}_{\xi,\xi-\xi'}(\Phi)=\Psi\cdot \gamma_{\xi-\xi'}$ is a constant loop.

\end{proof}
\begin{defn}
 We say that  $\zeta\in  \mathfrak{I}'({G}^{\sigma_\varrho})\cap \mathfrak{C}^\varrho_I$ is a $\varrho$-\emph{semi-canonical} element if  $\zeta$ is of the form $\zeta=\sum_{i\in I}n_i\zeta_i$ with $1\leq n_i\leq 2m_i$, where
 $m_i$ is the least positive integer which makes $m_i\zeta_i\in\mathfrak{I}'({G}^{\sigma_\varrho})$.
\end{defn}

\begin{cor}\label{crorol}
   Take $\xi\in \mathfrak{I}'({G}^{\sigma_\varrho})\cap \mathfrak{C}^\varrho_I$, with $I\subset\{1,\ldots,k\}$.  Let  $\Phi:S^2\setminus D\to U^{\sigma_\varrho}_\xi(G)$ be a $T_{\sigma_\varrho}$-invariant extended solution, and let $\varphi:S^2\to P_\xi^{\sigma_\varrho}$ be the corresponding harmonic map.
 Then  there exist $h\in G$, a constant loop $\gamma$, and  a $\varrho$-\emph{semi-canonical} $\zeta$
 such that $\tilde\Phi:=h\gamma\Phi h^{-1}$ defined on $S^2\setminus D$ takes values in $U^{\sigma_\varrho}_{\zeta}(G)$. The harmonic map $\tilde\Phi_{-1}$ takes values in $P^{\sigma_\varrho}_{\zeta}= P^{\sigma_\varrho}_{\xi}$ and coincides with $\varphi$ up to isometry.
\end{cor}
\begin{proof}Write $\xi=\sum_{i\in I}r_i\zeta_i$, with $r_i>0$. For each $i\in I$, let $n_i$ be the unique integer number in $\{1,\ldots, 2m_i\}$ such that $n_i=r_i\mod 2m_i.$
Set $\zeta=\sum_{i\in I}n_i\zeta_i$. It is clear that $\xi\preceq \zeta$ and $\zeta \in  \mathfrak{I}'({G}^{\sigma_\varrho})\cap \mathfrak{C}^\varrho_I$.
Observe also that conditions  \eqref{toing2} hold automatically for any $\xi'\in \mathfrak{I}'({G}^{\sigma_\varrho})\cap \mathfrak{C}^\varrho_I$ satisfying $\xi\preceq\xi'$. In particular they hold for $\xi'=\zeta$. Finally, since $\xi-\zeta=2\sum_{i\in I}m_ik_i\zeta_i$ for some nonnegative integer numbers $k_i$, then $\exp{\pi(\xi-\zeta)}=e$, and the result follows from Propositions \ref{norminf} and \ref{norm2}.   \end{proof}

\subsubsection{Classification of harmonic two-spheres into outer symmetric spaces}\label{classs}

To sum up, in order to classify all harmonic two-spheres into outer symmetric spaces we proceed as follows:
\begin{enumerate}
  \item Start with  a fundamental outer   involution $\sigma_\varrho$ and let $N$ be an outer symmetric $G$-space corresponding to an involution of the form $\sigma_\varrho$ or $\sigma_{\varrho,i}$ of $G$, according to  \eqref{sigmasis}, where the element $\zeta_i$ is in the conditions of Theorem \ref{murak}.  We assume that $\exp2\pi\zeta_i=e$, that is $\zeta_i\in \mathfrak{I}'({G}^{\sigma_\varrho})$. Let $\varphi:S^2\to N$ be an harmonic map and
identify $N$ with $P_{\zeta_i}^{\sigma_\varrho}=\exp(\pi\zeta_i)P_e^{\sigma_{\varrho,i}}$ via the totally geodesic embedding \eqref{tttt}. If $N$ is the fundamental outer space  with
involution $\sigma_\varrho$ we simply identify $N$ with  $P_e^{\sigma_\varrho}$ via $\iota_{\sigma_\varrho}$.

\item By Theorem \ref{tinva}, $\varphi:S^2\to N\cong P_{\zeta_i}^{\sigma_\varrho}$ admits a $T_{\sigma_\varrho}$-invariant extended solution $\Phi:S^2\to\Omega^{\sigma_\varrho} G$ which takes values, off some discrete subset $D$, in some unstable manifold $U_{\zeta'}^{\sigma_\varrho}(G)$, with $\zeta'\in\mathfrak{I}'({G}^{\sigma_\varrho})$; moreover,  $P_{\zeta'}^{\sigma_\varrho}=P_{\zeta_i}^{\sigma_\varrho}$.

    \item  By Corollary \ref{crorol}, we can assume that  $\zeta'$ is a  $\varrho$-semi-canonical element in $\mathfrak{I}'({G}^{\sigma_\varrho})\cap \mathfrak{C}^\varrho_I$.
    If $\zeta$ is a $\varrho$-canonical element such that $\zeta'\preceq \zeta$ and $\mathcal{U}_{\zeta',\zeta'-\zeta}(\Phi)$ is constant, then,
    taking into account Proposition \ref{norminf}, there exists  a $T_\tau$-invariant extended solution $\tilde\Phi: S^2\setminus D\to U^\tau_{\zeta}(G),$  where \begin{equation}\label{otau}\tau=\mathrm{Ad}(\exp \pi(\zeta'-\zeta))\circ \sigma_\varrho,\end{equation} such that the harmonic map $\varphi$  is given, up to isometry, by $\tilde\Phi_{-1}:S^2\to P_{\zeta}^\tau.$ Here we identify
 $N$  with $P_{\zeta}^\tau =\exp(\pi(\zeta'-\zeta)) P_{\zeta_i}^{\sigma_\varrho}$ via the composition of \eqref{tttt} with the left multiplication by $\exp(\pi(\zeta'-\zeta))$.

    \item By Proposition \ref{norm2}, there always exists a $\varrho$-canonical element $\zeta$ in such conditions.
    \end{enumerate}

Hence, we classify harmonic spheres into outer symmetric $G$-spaces in terms of pairs $(\zeta,\tau)$, where $\zeta$ is a  $\varrho$-canonical element and
$\tau$ is an outer involution of the form \eqref{otau} for some $\varrho$-semi-canonical element $\zeta'$ with $\zeta'\preceq \zeta$.

%\begin{rem}
%  Let $\tau$ and $\tau'$ be two conjugated outer involutions of the form \eqref{otau} and let $\zeta$ be a $\varrho$-canonical element. We have $$\tau'=Ad(h)\circ \tau \circ Ad(h^{-1})$$ for some $h\in G$.  In general, even $\tau$ and $\tau'$ being  conjugated, $P^\tau_\varrho$ and $P^{\tau'}_{\varrho}$ are representatives of different isometry classes of symmetric spaces and, consequently,  we will consider the pairs $(\zeta,\tau)$ and $(\zeta,\tau')$  as different classes of harmonic maps, unless $\tau$ is the fundamental involution $\sigma_\varrho$. In this particular case, if  $\Phi:S^2\to U_\zeta^{\sigma_\varrho}$ is an extended solution, then we have:  $\Phi_{-1}$ takes values in $P_\zeta^{\sigma_\varrho}$; the extended solution $\Phi'=h\Phi h^{-1}$ takes values in $U_\zeta^{\tau'}$ and the corresponding harmonic map $\Phi_{-1}'=h\Phi_{-1}h^{-1}$ takes values in $P_\zeta^{\tau'}\cong P_{\zeta}^{\sigma_\varrho}$; since $\Phi_{-1}$ and $\Phi_{-1}'$ coincide up to isometry, $(\zeta,\sigma_\varrho)$ and $(\zeta,\tau')$   define the same isometry class of harmonic maps.\end{rem}

\subsubsection{Weierstrass Representation for $T_\sigma$-invariant Extended Solutions}

From \eqref{ondepsi} and Proposition \ref{BG}, we obtain the  following.

\begin{thm}\label{sigmaweirstrass}
  Let $\Phi:M\to \Omega_{\mathrm{alg}}^\sigma G$ be an extended solution.
  There exists a discrete set $D'\supseteq D$ of $M$  such that
  $\Phi{\big|_{M\setminus D'}}=\exp C\cdot \gamma_\xi$ for some  holomorphic vector-valued function $C: M\setminus D'\to (\mathfrak{u}^0_\xi)_\sigma$,
   where $(\mathfrak{u}^0_\xi)_\sigma$ is the finite dimensional nilpotent subalgebra of $\Lambda^+_{\mathrm{alg}}\mathfrak g^\C$ defined by
$$(\mathfrak{u}^0_\xi)_\sigma=\bigoplus_{0\leq 2i<r(\xi)}\lambda^{2i}(\mathfrak{p}^\xi_{2i})^\perp\cap \mathfrak{k}^\C_\sigma\oplus \bigoplus_{0\leq 2i+1<r(\xi)}\lambda^{2i+1}(\mathfrak{p}^\xi_{2i+1})^\perp\cap \mathfrak{m}^\C_\sigma,$$
with $(\mathfrak{p}^\xi_i)^\perp=\bigoplus _{i<j\leq r(\xi)}\mathfrak{g}_j^\xi$.
 Moreover, $C$ can be extended meromorphically to $M$.

  \end{thm}

%\begin{prop}
%Let $N=G/K$ be an outer symmetric space and $\varphi:M\to N$ an harmonic map associated to the canonical element $\xi$. If $\mathfrak{g}^\alpha\subset \mathfrak{k}^\C$ for any positive root $\alpha$ with $\alpha(\xi)$ even, then $\varphi$ takes values in a proper totally geodesic submanifold of $N$.
%\end{prop}
%\begin{proof}
%
%\end{proof}
%
%
%\begin{thm}
%  Let $N=G/K$ be an outer symmetric space and $\varphi:M\to N$ an harmonic map. Assume that
%  $$\Delta':=\{\alpha\in\Delta |\, \mathfrak{g}_\alpha\nsubseteq \mathfrak{m}^\C\}.$$
%  Then $\varphi$ takes values in some inner $K$-symmetric space totally geodesically embedded
%  in $N$.
%\end{thm}

\section{Examples}\label{examples}

Next we will describe explicit examples of harmonic spheres into \emph{classical} outer symmetric spaces.

\subsection{Outer symmetric $SO(2n)$-spaces}
 For details on the structure of $\mathfrak{so}(2n)$ see \cite{fulton_harris}. Consider on $\mathbb{R}^{2n}$ the standard inner product $\langle \cdot, \cdot \rangle$ and fix a complex basis $\mathbf{u}=\{u_1,\ldots,u_n,\overline{u}_1,\ldots,\overline{u}_n\}$ of $\mathbb{C}^{2n}=(\mathbb{R}^{2n})^\mathbb{C}$
satisfying
\begin{equation*}\label{us}
  \langle u_i,u_j\rangle=0, \quad  \langle u_i,\overline{u}_j\rangle=\delta_{ij},\quad \mbox{for all $1\leq i,j\leq n$}.
\end{equation*}
 Throughout this section we will denote by  $V_l$ the $l$-dimensional isotropic subspace spanned by $\overline{u}_1,\ldots,\overline{u}_l$.

 %
% real orthonormal basis $\mathbf{e}=\{e_1,\ldots,e_{2n}\}$ of $\mathbb{R}^{2n}$. Define the complex orthonormal basis $\mathbf{u}=\{u_1,\ldots,u_n,\overline{u}_1,\ldots,\overline{u}_n\}$ of $\mathbb{C}^{2n}$ by $$\mbox{$u_j=\frac{1}{\sqrt2}(e_j+\mathrm{i}e_{n+j})$}.$$
 Set $E_i=E_{i,i}-E_{n+i,n+i}$, where $E_{j,j}$ is a square matrix, with respect to the basis $\mathbf{u}$, whose $(j,j)$-entry is $\mathrm{i}$ and all other entries are $0$. The complexification $\lt^\C$ of the algebra $\lt$ of diagonal matrices  $\sum a_iE_i$,  with $a_i\in \mathbb{R}$ and $\sum a_i=0$,
    is a Cartan subalgebra of $\mathfrak{so}(2n)^\C$. Let $\{L_1,\ldots,L_n\}$ be the dual basis in $\mathrm{i}\lt^*$ of $\{E_1,\ldots,E_n\}$, that is  $L_i(E_j)=\mathrm{i}\delta_{ij}$. The roots of $\mathfrak{so}(2n)$ are the vectors $\pm L_i\pm L_j$ and $\pm L_i\mp L_j$, with $i\neq j$ and $1\leq i,j\leq n$.

 Consider the endomorphisms
 \begin{equation}\label{geradores}
 X_{i,j}=E_{i,j}-E_{n+j,n+i},\,\,Y_{i,j}=E_{i,n+j}-E_{j,n+i},\,\, Z_{i,j}=E_{n+i,j}-E_{n+j,i},
 \end{equation}
 where $E_{i,j}$, with $i\neq j$, is a square matrix whose $(i,j)$-entry is $1$ and all other entries are $0$.
 The root spaces of $L_i-L_j$, $L_i+L_j$ and $-L_i-L_j$, respectively, are  generated by the endomorphisms $X_{i,j}$, $Y_{i,j}$ and  $Z_{i,j}$, respectively.

  Fix the positive root system
$\Delta^+=\{L_i\pm L_j\}_{i<j}.$ The positive simple roots are  $\alpha_i=L_i-L_{i+1}$, for $1\leq i\leq n-1$, and $\alpha_n=L_{n-1}+L_n$. The vectors of the dual basis $\{H_1,\ldots,H_n\}\subset \lt$  are given by $H_i  =E_1+E_2+\ldots+E_i$, for $1\leq i\leq n-2$,
$$\mbox{$H_{n-1}  =\frac12(E_1+E_2+\ldots+E_{n-1}-E_n)$, and $H_n  =\frac12(E_1+E_2+\ldots+E_{n-1}+E_n)$}.$$

Consider the non-trivial involution $\varrho$ of the corresponding Dynkin diagram,
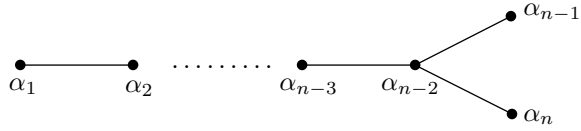
\begin{center}
\psset{xunit=1.5cm,yunit=1.3cm,algebraic=true,dotstyle=o,dotsize=2pt 0,linewidth=0.5pt,arrowsize=5pt 2,arrowinset=0.25}
\begin{pspicture*}(-0.9,1.11)(4.85,2.76)
\psline(0,0)(1,0)
\psline(1,0)(2,0)
\psline(2,0)(3,0)
\psline(3,0)(4,0)
\psline(2,0)(1.95,0.21)
\rput[tl](-0.6,1.86){$\alpha_1$}
\rput[tl](0.43,1.84){$\alpha_2$}
\rput[tl](0.92,-0.13){$\alpha_3$}
\rput[tl](1.8,1.86){$\alpha_{n-3}$}
\rput[tl](2.7,1.86){$ \alpha_{n-2}$}
\rput[tl](3.97,-0.13){$\alpha_6$}
\psline(-0.5,2)(0.5,2)
\psline(3,2)(3.86,1.5)
\psline(3,2)(3.85,2.5)
\psline(2,2)(3,2)
\rput[tl](0.84,2){$\ldots\ldots\ldots$}
\rput[tl](3.96,1.55){$\alpha_n$}
\rput[tl](3.96,2.63){$\alpha_{n-1}$}
\begin{scriptsize}
\psdots[dotsize=4pt 0,dotstyle=*](0,0)
\psdots[dotsize=4pt 0,dotstyle=*](1,0)
\psdots[dotsize=4pt 0,dotstyle=*](2,0)
\psdots[dotsize=4pt 0,dotstyle=*](3,0)
\psdots[dotsize=4pt 0,dotstyle=*](4,0)
\psdots[dotsize=4pt 0,dotstyle=*](1.95,0.21)
\psdots[dotstyle=*,linecolor=blue](0,4)
%\rput[bl](-4.03,3.01){\blue{$G$}}
\psdots[dotsize=4pt 0,dotstyle=*](-0.5,2)
\psdots[dotsize=4pt 0,dotstyle=*](0.5,2)
\psdots[dotsize=4pt 0,dotstyle=*](3,2)
\psdots[dotsize=4pt 0,dotstyle=*](3.86,1.5)
\psdots[dotsize=4pt 0,dotstyle=*](3.85,2.5)
\psdots[dotstyle=*,linecolor=blue](-2.26,0.91)
%\rput[bl](-2.22,0.96){\blue{$M$}}
\psdots[dotsize=4pt 0,dotstyle=*](2,2)
\end{scriptsize}
\end{pspicture*}
\end{center}
This involution  fixes $\alpha_i$ if $i\leq n-2$ and $\varrho(\alpha_{n-1})= \alpha_n$.
The corresponding semi-fundamental basis
$\pi_{\mathfrak{k}_\varrho}(\Delta_0)=\{\beta_1,\ldots,\beta_{n-1}\}$ is given by
$$\mbox{$\beta_i=\alpha_i=L_i-L_{i+1}$, if $i\leq n-2$, and $\beta_{n-1}=\frac12(\alpha_{n-1}+\alpha_{n})=L_{n-1}$},$$
whereas the dual basis $\{\zeta_1,\ldots,\zeta_{n-1}\} $ is given by
$\zeta_i=E_1+\ldots+E_i$,
with $i=1,\ldots,n-1$. Since each $\zeta_i$ belongs to the integer lattice $\mathfrak{I}(SO(2n)^{\sigma_\varrho})$, we have:
\begin{prop}
  The $\varrho$-semi-canonical elements of $SO(2n)$ are precisely the elements $\zeta=\sum_{i=1}^{n-1} m_i\zeta_i$ such that $m_i\in\{0,1,2\}$ for $1\leq i\leq n-1$.
\end{prop}

The fundamental outer symmetric $SO(2n)$-space is the real projective space $\mathbb{R}P^{2n-1}$, and the associated outer symmetric $SO(2n)$-spaces are the real Grassmannians $G_p(\mathbb{R}^{2n})$ with $p>1$ odd.

 \subsubsection{Harmonic maps into real projective spaces $\mathbb{R}P^{2n-1}$.}
Consider as base point  the one dimensional real vector space  $V_0$ spanned by $e_n=(u_n+\overline{u}_n)/\sqrt{2}$ in $\mathbb{R}^{2n}$, which establishes an identification of $\mathbb{R}P^{2n-1}$
with $SO(2n)/S(O(1)O(2n-1)).$
Denote by $\pi_{V_0}$ and $\pi_{V_0}^\perp$  the orthogonal projections onto $V_0$ and $V_0^\perp$, respectively. The  fundamental involution is given by $\sigma_\varrho=\mathrm{Ad}(s_0)$, where  $s_0=\pi_{V_0}-\pi_{V_0}^\perp$.
Following the classification procedure established in Section \ref{classs}, we start by identifying
 $\mathbb{R}P^{2n-1}$ with $P_e^{\sigma_\varrho}$.

\begin{thm}\label{classesproj}
  Each harmonic map  $\varphi:S^2\to\mathbb{R}P^{2n-1}$ belongs to one of the following classes: $(\zeta_l,\sigma_{\varrho,l})$, with $1\leq l\leq n-1$.
\end{thm}
\begin{proof}Let $\zeta$ be a  $\varrho$-semi-canonical element and write
\begin{equation}\label{ozeta}\zeta=\sum_{i\in I_1}\zeta_i+ \sum_{i\in I_2}2\zeta_i\end{equation} for some disjoint subsets $I_1$ and $I_2$ of $\{1,\ldots,n-1\}$.
 By Proposition \ref{fundcan}, $P^{\sigma_\varrho}_\zeta\cong \mathbb{R}P^{2n-1}$ if and only if either $I_1=\emptyset$ or $I_1=\{n-1\}$.
Suppose that $I_1=\{n-1\}$. In this case,
$\exp\pi\zeta=\exp\pi\zeta_{n-1}  \in P_{\zeta_{n-1}}^{\sigma_\varrho}$.
 We claim that $P_{\zeta_{n-1}}^{\sigma_\varrho}$ is not the connected component of $P^{\sigma_\varrho}$ containing the identity $e$.
 Write $\exp\pi\zeta_{n-1}=\pi_V-\pi_V^\perp,$ where $V$ is the two-dimensional real space spanned by $e_{n}$ and $e_{2n}$. For each $g\in  P_e^{\sigma_\varrho}$, since the $G$-action $\cdot_{\sigma_\varrho}$  defined by \eqref{gaction} is transitive, we have $g=h\cdot_{\sigma_\varrho}e=hs_0h^{-1}s_0$ for some $h\in G$, which means that $gs_0=hs_0h^{-1}$. In particular, the $+1$-eigenspaces of $gs_0$  must be $1$-dimensional.  However, a simple computation shows that the $+1$-eigenspace of $\exp(\pi\zeta_{n-1})s_0$ is $3$-dimensional, which establishes our claim.

Then, any harmonic map $\varphi:S^2\to\mathbb{R}P^{2n-1}\cong P_e^{\sigma_\varrho}$ admits a $T_{\sigma_\varrho}$-invariant extended solution $\Phi:S^2\setminus D\to U^{\sigma_\varrho}_\zeta(SO(2n))$ with $\zeta$ a  $\varrho$-semi-canonical element of the form
$\zeta=\sum_{i\in I_2}2\zeta_i.$ Set $l=\max I_2$. Next we check that $\zeta$ and $\zeta_l$ satisfy the conditions of Proposition \ref{norm2}, with $\xi=\zeta$ and $\xi'=\zeta_l$.
It is clear that $\zeta\preceq \zeta_l$. Now, according to \eqref{deltas} and \eqref{delta1}, we can take $\Delta'_\varrho=\{L_i-L_n, L_n-L_i\}$.  Hence, for $i>0$,
\begin{equation*}
\mathfrak{g}_{2i}^\zeta\cap \mathfrak{m}_\varrho^\C=\bigoplus_{\alpha\in \Delta'_\varrho\cap \Delta_\zeta^{2i}}(\g_\alpha\oplus \g_{\varrho(\alpha)})\cap\mathfrak{m}^\C_\varrho,
\end{equation*}
where $\Delta_\zeta^{2i}=\{\alpha\in \Delta|\, \alpha(\zeta)=2i\mathrm{i}\}$. Since $$(L_j-L_n)(\zeta)=(\alpha_j+\alpha_{j+1}+\ldots+\alpha_{n-1})(\zeta)=2|I_2\cap \{j,\ldots,n-1\}|\mathrm{i},$$ we have
$$\Delta'_\varrho\cap \Delta_\zeta^{2i}=\{L_j-L_n|\,\, \mbox{$1\leq j\leq l$, and  $|I_2\cap\{j,\ldots,l\}|=i$}\}.$$
Then, given a root $\alpha=L_j-L_n\in \Delta'_\varrho\cap \Delta_\zeta^{2i}$ (in particular, $j\leq l$) we have
$\alpha(\zeta-\zeta_l)=(2i-1)\mathrm{i},$
which means that $\g_\alpha\subset \g^{\zeta-\zeta_l}_{2i-1}$. Consequently,
 $$\g_{2i}^\zeta\cap \mathfrak{m}_\varrho^\C  \subset \bigoplus_{0\leq j<2i}\g^{\zeta-\zeta_l}_j.$$
Since $\g_{2i-1}^\zeta=\{0\}$ for all $i$, we conclude that \eqref{toing2} holds, and the statement  follows from Propositions \ref{norminf} and \ref{norm2}.
\end{proof}
It is known \cite{calabi_1967} that there are no full harmonic maps $\varphi:S^2\to\mathbb{R}P^{2n-1}$. The class of harmonic maps associated to $(\zeta_l,\sigma_{\varrho,l})$ consists precisely of those $\varphi$ with $\varphi(S^2)$ contained, up to isometry, in   some $\mathbb{R}P^{2l}$, as shown in the next theorem.

\begin{thm}\label{RAs}
  Given $1\leq l\leq n-1$, any harmonic map  $\varphi:S^2\to\mathbb{R}P^{2n-1}$ in the class $(\zeta_l,\sigma_{\varrho,l})$ is given by
  \begin{equation}\label{projspace}
  \varphi={R}\cap(A\oplus \overline{A})^\perp,
  \end{equation}
   where ${R}$ is a constant $2l+1$-dimensional subspace of $\mathbb{R}^{2n}$ and $A$
  is a holomorphic isotropic subbundle of $S^2\times {R}$ of rank $l$ satisfying $\partial A\subseteq \overline{A}^\perp$. The corresponding extended solutions have uniton number $2$ with respect to the standard representation of $SO(2n)$.
  \end{thm}
  \begin{proof}
  Let  $\varphi:S^2\to\mathbb{R}P^{2n-1}$ be a harmonic map in the class $(\zeta_l,\sigma_{\varrho,l})$. This means that $\varphi$ admits an extended solution $\Phi:S^2\setminus D \to U_{\zeta_l}^{\sigma_{\varrho,l}}(SO(2n))$. Up to isometry, $\varphi$ is given by $\Phi_{-1}$, which takes values in $P_{\zeta_l}^{\sigma_{\varrho,l}}=\exp(\pi \zeta_l)P_e^{\sigma_{\varrho}}$. This connected component is identified with $\mathbb{R}P^{2n-1}$ via
 \begin{equation}\label{identi}
 g\cdot V_0\mapsto \exp(\pi \zeta_l)g\sigma_\varrho(g^{-1}).
 \end{equation}
Write
$\gamma_{\zeta_l}(\lambda)=\lambda^{-1}\pi_{V_l}+\pi_{V_l\oplus\overline{V}_l}^\perp+\lambda \pi_{\overline{V}_l},$
where $V_l$ is the $l$-dimensional isotropic subspace spanned by $\overline{u}_1,\ldots,\overline{u}_l$.

 We have $r(\zeta_l)=2$ if $l>1$ and $r(\zeta_1)=1$. Consequently, by Theorem \ref{sigmaweirstrass},
$$(\mathfrak{u}^0_{\zeta_l})_{\sigma_{\varrho,l}}=(\mathfrak{p}^{\zeta_l}_{0})^\perp\cap \mathfrak{k}^\C_{\sigma_{\varrho,l}}\oplus \lambda (\mathfrak{p}^{\zeta_l}_{1})^\perp\cap \mathfrak{m}^\C_{\sigma_{\varrho,l}}.$$
Here $(\mathfrak{p}^{\zeta_l}_{1})^\perp=\g_2^{\zeta_l}$, which is the null space for $l=1$. For $l>1$, since $\zeta_l=E_1+\ldots+E_l$, we have  $\g_2^{\zeta_l}=\{L_i+L_j|\, 1\leq i< j\leq l\}\subset \Delta(\mathfrak{k}_\varrho)$ and, from \eqref{pc},
$$\mathfrak{m}_{\sigma_{\varrho,l}}^\C=\bigoplus \mathfrak{g}^{\zeta_l}_{2i+1}\cap \mathfrak{k}^\C_{\varrho}\oplus\bigoplus \mathfrak{g}^{\zeta_l}_{2i}\cap \mathfrak{m}^\C_{\varrho}.$$
Hence
$(\mathfrak{p}^{\zeta_l}_{1})^\perp\cap \mathfrak{m}^\C_{\sigma_{\varrho,l}}= \g_2^{\zeta_l}\cap \mathfrak{m}^\C_{\varrho}=\{0\}.$
Then, for any $l\geq 1$,  we can write $\Phi=\exp C\cdot \gamma_{\zeta_l}$ for some holomorphic function
$$C:S^2\setminus D\to (\mathfrak{p}^{\zeta_l}_{0})^\perp\cap\mathfrak{k}^\C_{\sigma_{\varrho,l}} =(\g_1^{\zeta_l}\oplus\g_2^{\zeta_l})\cap \mathfrak{k}^\C_{\sigma_{\varrho,l}},$$
which means that $\Phi$ is a $S^1$-invariant extended solution with uniton number $2$:
\begin{equation}\label{phiw}
\Phi_\lambda=\lambda^{-1}\pi_{W}+\pi_{W\oplus\overline{W}}^\perp+\lambda \pi_{\overline{W}},
\end{equation}
where $W$ is a holomorphic isotropic subbundle of $S^2\times \mathbb{R}^{2n}$ of rank $l$ satisfying the superhorizontality condition $\partial W\subseteq \overline{W}^\perp$.

 Set $\tilde{V}_l= V_l\oplus \overline{V}_l$ and $\tilde{W}= W\oplus \overline{W}$.  The $T_{{\sigma_{\varrho,l}}}$-invariance of $\Phi$ implies that
 \begin{equation}\label{inva} [\pi_W,\pi_{V_0\oplus \tilde V_l}]=0.\end{equation}
Now, write $\varphi=g\cdot V_0$ and consider the identification \eqref{identi}. We must have
\begin{equation}\label{phirp}
  \Phi_{-1}= \exp(\pi \zeta_l)g\sigma_\varrho(g^{-1})=\exp(\pi \zeta_l)(\pi_\varphi-\pi_\varphi^\perp)s_0.
\end{equation}
From \eqref{phiw} and \eqref{phirp}  we obtain
\begin{equation}\label{pias}
\pi_\varphi-\pi_\varphi^\perp=\mathrm{Ad}(s_0)\big(\pi_{V_0\oplus \tilde V_l}\pi_{\tilde W}^\perp+\pi_{V_0\oplus \tilde V_l}^\perp\pi_{\tilde W}
-\pi_{V_0\oplus \tilde V_l}\pi_{\tilde W}-\pi_{V_0\oplus \tilde V_l}^\perp\pi_{\tilde W}^\perp\big).\end{equation}
In view of \eqref{inva}, we see that
$\pi_{V_0\oplus \tilde V_l}\pi_{\tilde W}^\perp+\pi_{V_0\oplus \tilde V_l}^\perp\pi_{\tilde W}$
is an orthogonal projection, and \eqref{pias} implies that this must be an orthogonal projection onto a $1$-dimensional real subspace. Then, one of its two terms vanishes, that is  either $\tilde W\subset V_0\oplus \tilde V_l$ or $\tilde W^\perp\subset (V_0\oplus \tilde V_l)^\perp$. For dimensional reasons, we see that the second case can not occur. Hence, we have
$$\pi_\varphi=\mathrm{Ad}(s_0)( \pi_{V_0\oplus \tilde V_l}\pi_{\tilde W}^\perp)= \pi_{V_0\oplus \tilde V_l}\mathrm{Ad}(s_0)(\pi_{\tilde W}^\perp),$$
that is \eqref{projspace} holds with ${R}=V_0\oplus  V_l\oplus \overline V_l$ and $A=s_0(W)$.

  \end{proof}
\begin{rem}
  If $\varphi$ is full in ${R}$, then the isotropic subbundle $A$ is the $l$-osculating space of some full totally isotropic holomorphic map $f$ from $S^2$ into the complex projective space of ${R}$, the so called \emph{directrix curve}  of $\varphi$. That is, in a local system of coordinates $(U,z)$, we have
$A(z)=\mathrm{Span}\big\{g,g',\ldots,g^{(l-1)}\}$, where $g$ is a lift of $f$ over $U$ and $g^{(r)}$ the $r$-th derivative of $g$ with respect to $z$. Hence, formula \eqref{projspace} agrees with
the classification given in Corollary 6.11 of \cite{eells_wood_1983}.
\end{rem}

\begin{eg}Let us consider the case $n=2$. We have only one class of harmonic maps:  $(\zeta_1,\sigma_{\varrho,1})$. From Theorem \ref{RAs},
  any such harmonic map $\varphi:S^2\to\R P^3$ is given by $\varphi=R\cap(A\oplus\overline{A})^\perp$, where $R$ is a constant 3-dimensional subspace of $\R^4$ and $A$ a holomorphic isotropic subbundle  of $S^2\times R$ of rank 1 such that $\partial A\subseteq\overline{A}^\perp$. Taking into account Theorem \ref{sigmaweirstrass}, any such holomorphic subbundles $A$ can be obtained from  a meromorphic function $a$ on $S^2$ as follows.

We have $\zeta_1=E_1$ and the corresponding extended solutions have uniton number $r(\zeta_1)=1$ (with respect to the standard representation). Any  extended solution $\Phi:S^2\setminus D\to U_{\zeta_1}^{\sigma_{\varrho,1}}(SO(4))$ is given by $\Phi=\exp C\cdot\gamma_{\zeta_1}$, with $\gamma_{\zeta_1}(\lambda)=\lambda^{-1}\pi_{{V}_1}+\pi_{V_1\oplus\overline{V}_1}^\perp+\lambda\pi_{\overline{V}_1}$, for some holomorphic vector-valued function $C:S^2\setminus D\to(\mathfrak{u}_{\zeta_1}^0)_{\sigma_{\varrho,1}}$, where $$(\mathfrak{u}_{\zeta_1}^0)_{\sigma_{\varrho,1}}=(\mathfrak{p}_0^{\zeta_1})^\perp\cap\mathfrak{k}^\C_{\sigma_{\varrho,1}}=\mathfrak{g}_1^{\zeta_1}\cap\mathfrak{k}^\C_{\sigma_{\varrho,1}}
=(\mathfrak{g}_{L_1-L_2}\oplus\mathfrak{g}_{L_1+L_2})\cap\mathfrak{k}^\C_{\sigma_{\varrho,1}}.$$

Considering the   root vectors $X_{i,j},Y_{i,j},Z_{i,j}$ as defined in \eqref{geradores}, we have $Y_{1,2}=\sigma_{\varrho,1}(X_{1,2})$.
Hence $C=a(z)(X_{1,2}+Y_{1,2})$ where $a(z)$ is a meromorphic function on $S^2$.
In this case, from \eqref{C_z}, it follows that $(\exp C)^{-1}(\exp C)_z=C_z$, and  it is clear that the extended solution condition for $\Phi$ holds independently of the choice of the meromorphic function $a(z)$.
Then, with respect to the complex basis $\mathbf{u}=\{u_1,u_2,\overline{u}_1,\overline{u}_2\}$,
\begin{equation}\label{zeta1rp3}\exp C\cdot\gamma_{\zeta_1}=\left[\begin{array}{cccc} 1& a & -a^2 & a\\ 0 & 1 & -a & 0\\ 0 & 0 & 1 & 0\\ 0 & 0 & -a & 1 \end{array}\right]\cdot\gamma_{\zeta_1}\end{equation}
and the subbundle $A$ of $R=\mathrm{Span}\{{u}_1,\overline{u}_1, u_2+\overline{u}_2\}$ is given by $A=\exp C\cdot {V}_1=\mathrm{span}\{(a^2,a,-1,a)\}$, which satisfies $\partial A\subseteq\overline{A}^\perp$.
\end{eg}

\begin{eg}Any harmonic two-sphere into $\R P^5$ in the class $(\zeta_1,\sigma_{\varrho,1})$ takes values in some $\R P^3$ inside $\R P^5$ and so it is essentially of the form (\ref{zeta1rp3}).
Next we consider the Weierstrass representation of harmonic spheres into $\R P^5$ in the class $(\zeta_2,\sigma_{\varrho,2})$, which are given by $\varphi=R\cap(A\oplus\overline{A})^\perp$, where $R$ is a constant 5-dimensional subspace of $\R^6$ and $A$ a holomorphic isotropic subbundle  of $S^2\times R$ of rank 2 such that $\partial A\subseteq\overline{A}^\perp$.
We have $\zeta_2=E_1+E_2$, then $r(\zeta_2)=2$.
Any extended solution $\Phi:S^2\setminus D\to U_{\zeta_2}^{\sigma_{\varrho,2}}(SO(6))$ is given by $\Phi=\exp C\cdot\gamma_{\zeta_2}$, with $\gamma_{\zeta_2}(\lambda)=\lambda^{-1}\pi_{{V}_2}+\pi_{V_2\oplus\overline{V}_2}^\perp+\lambda\pi_{\overline{V}_2},$ for some holomorphic vector-valued function $C:S^2\setminus D\to(\mathfrak{u}_{\zeta_2}^0)_{\sigma_{\varrho,2}}$, where
\begin{align*} (\mathfrak{u}_{\zeta_2}^0)_{\sigma_{\varrho,2}}  =\left((\mathfrak{g}_{L_1- L_3}\oplus \mathfrak{g}_{L_1+L_3})\cap\mathfrak{k}^\C_{\sigma_{\varrho,2}}\right)\oplus\left((\mathfrak{g}_{L_2- L_3}\oplus \mathfrak{g}_{L_2+L_3})\cap\mathfrak{k}^\C_{\sigma_{\varrho,2}}\right)\oplus \mathfrak{g}_{L_1+L_2}. \end{align*}

We have $Y_{1,3}=\sigma_{\varrho,2}(X_{1,3})$ and $Y_{2,3}=\sigma_{\varrho,2}(X_{2,3})$. Hence we can write
 $$C=a(z)(X_{1,3}+Y_{1,3})+b(z)(X_{2,3}+Y_{2,3})+c(z)Y_{1,2}$$
 where $a(z)$, $b(z)$ and $c(z)$ are meromorphic functions on $S^2$.

 Now, $\Phi=\exp{C}\cdot \gamma_{\zeta_2}$ is an
 extended solution if and only if, in the expression
$C_z-\frac{1}{2!}(\mathrm{ad} C)C_z,$
which does not depend on $\lambda$, the component on $\g_2^{\zeta_2}=\g_{L_1+L_2}$  vanishes. Since
$Y_{1,2}=[Y_{2,3},X_{1,3}]=[X_{2,3},Y_{1,3}]$ and $[X_{1,3},X_{2,3}]=[Y_{1,3},Y_{2,3}]=0$,
 this holds if and only if
$c'=ba'-ab'$, where prime denotes $z$-derivative. Since $A=\exp C\cdot V_2$, we can compute $\exp C$ in order to conclude that the holomorphic subbundle $A$ of $R=\mathrm{Span}\{{u}_1,u_2,\overline{u}_1,\overline{u}_2, u_3+\overline{u}_3\}$ is given by
$$A=\mathrm{Span}\{(a^2,ab+c,a,-1,0,a),(ab-c,b^2,b,0,-1,b)\}.$$
\end{eg}

\subsubsection{Harmonic maps into Real Grassmanians.}

% We will assume that $p\leq n$ since we can identify $G_p(\mathbb{R}^{2n})$ with $G_{2n-p}(\mathbb{R}^{2n})$ via $V\mapsto V^\perp$.

  Let $\zeta'$ be a $\varrho$-semi-canonical element of $SO(2n)$ given by \eqref{ozeta}, for
some disjoint subsets $I_1$ and $I_2$ of $\{1,\ldots,n-1\}$. By Proposition \ref{fundcan}, we know that $P^{\sigma_\varrho}_{\zeta'}\cong \mathbb{R}P^{2n-1}$ if and only if either $I_1=\emptyset$ or $I_1=\{n-1\}$. More generally we have: \begin{prop}\label{grassd}
If $I_1=\{i_1>i_2>\ldots>i_r\}$ and $d=\sum_{j=1}^r(-1)^{j+1}i_j$, then $P^{\sigma_\varrho}_{\zeta'}\cong G_{2d+1}(\mathbb{R}^{2n})$.
\end{prop}
\begin{proof}
For $\zeta'$ of the form \eqref{ozeta}, set $\zeta'_{I_1}=\sum_{i\in I_1}\zeta_i$.
Clearly, $\exp\pi\zeta'=\exp\pi\zeta'_{I_1}$, and, by Proposition \ref{concomp}, $P^{\sigma_\varrho}_{\zeta'}$ is a symmetric space with involution
 $$\tau=\mathrm{Ad}(\exp\pi\zeta'_{I_1})\circ \sigma_\varrho=\mathrm{Ad}(s_0\exp \pi \zeta'_{I_1}).$$ We have
 $$\zeta'_{I_1}=r(E_1+\ldots + E_{i_r})+(r-1)(E_{i_r+1}+\ldots +E_{i_{r-1}})+\ldots +(E_{i_2+1}+\ldots +E_{i_1}),$$
and consequently, with the convention $V_{i_0}=V_n$ and $V_{i_{r+1}}=\{0\}$,
$$\exp\pi\zeta'_{I_1}=\sum_{j=0}^r(-1)^j\pi_{i_j-i_{j+1}}+\sum_{j=0}^r(-1)^j\overline{\pi}_{i_j-i_{j+1}},$$
where $\pi_{i_j-i_{j+1}}$ is the orthogonal projection onto $V_{i_j}\cap V_{i_{j+1}}^\perp$ and $\overline{\pi}_{i_j-i_{j+1}}$ the orthogonal projection onto the corresponding conjugate space. Hence, the $+1$-eigenspace of $s_0\exp\zeta'_{I_1} $ has dimension $2d+1$, with $d=\sum_{j=1}^r(-1)^{j+1}i_j$, which means that $P^{\sigma_\varrho}_{\zeta'}\cong G_{2d+1}(\mathbb{R}^{2n})$.
\end{proof}

In particular, we have $P_{\zeta_d}^{\sigma_\varrho}\cong G_{2d+1}(\mathbb{R}^{2n})$ for each $d\in\{1,\ldots,n-1\}$.

\begin{thm}\label{cangrass}
 Each harmonic map from $S^2$ into the real Grassmannian $G_{2d+1}(\mathbb{R}^{2n})$ belongs to one of the following classes:
 $(\zeta,\mathrm{Ad}\exp\pi(\tilde\zeta-\zeta)\circ \sigma_{\varrho,l}),$
 where $\zeta$ and $\tilde \zeta$ are $\varrho$-canonical elements such that $\tilde\zeta\preceq \zeta$ and $\tilde\zeta=\sum_{i\in I_1}\zeta_i+\zeta_l$, where
  \begin{enumerate}
    \item[a)] $I_1=\{i_1>i_2>\ldots>i_r\}$ satisfies $d=\sum_{j=1}^r(-1)^{j+1}i_j$;
    \item[b)] $l\in\{0,1,\ldots,n-1\}$ and $l\notin I_1$ (if $l=0$, we set $\zeta_0=0$).
  \end{enumerate}
\end{thm}

 \begin{proof}
We consider harmonic maps into $P_{\zeta_d}^{\sigma_\varrho}\cong G_{2d+1}(\mathbb{R}^{2n})$.
 Let $\zeta'$ be a  $\varrho$-semi-canonical element and write
$\zeta'=\sum_{i\in I_1}\zeta_i+ \sum_{i\in I_2}2\zeta_i$ for some disjoint subsets $I_1$ and $I_2$ of $\{1,\ldots,n-1\}$.
 By Proposition \ref{grassd}, $P^{\sigma_\varrho}_{\zeta'}\cong G_{2d+1}(\mathbb{R}^{2n})$ if and only if either $d=\sum_{j=1}^r(-1)^{j+1}i_j$
 or $n-d-1=\sum_{j=1}^r(-1)^{j+1}i_j$, since $G_{2d+1}(\mathbb{R}^{2n})$ and $G_{2d'+1}(\mathbb{R}^{2n}),$
 with $d'=n-d-1$, can be identified via $V\mapsto V^\perp$. However, it follows from the same reasoning  as in the proof of Theorem \ref{classesproj}  that, in the second case,  $P^{\sigma_\varrho}_{\zeta'}$ does not coincide with the connected component $P_{\zeta_d}^{\sigma_\varrho}$. So we only consider the $\varrho$-semi-canonical  elements $\zeta'$ with $d=\sum_{j=1}^r(-1)^{j+1}i_j$.

Set $l=\max I_2$.
Next we check that the pair $\zeta'\preceq \tilde\zeta=\sum_{i\in I_1}\zeta_i+\zeta_l$ satisfies the conditions of Proposition \ref{norm2}. Considering the same notations we  used in the proof  of Theorem \ref{classesproj}, for each $i>0$ we have
$$\Delta'_\varrho\cap \Delta_{\zeta'}^{2i}=\{L_j-L_n|\,\,\,   2|I_2\cap\{j,\ldots,l\}|+|I_1\cap\{j,\ldots,n-1\}|=2i\}.$$
In particular,  for $i>0$ and $\alpha=L_j-L_n\in \Delta'_\varrho\cap \Delta_{\zeta'}^{2i}$, it is clear that $\alpha(\zeta'-\tilde \zeta)/\mathrm{i}\leq 2i-1,$ and consequently
 $$\g_{2i}^{\zeta'}\cap \mathfrak{m}_\varrho^\C  \subset \bigoplus_{0\leq j<2i}\g^{\zeta'-\tilde\zeta}_j.$$

For $i>0$, we have the decomposition
$$\g_{2i-1}^{\zeta'}\cap \mathfrak{k}_\varrho^\C=\!\!\! \bigoplus_{{\alpha\in \Delta(\mathfrak{k}_\varrho)\cap \Delta_{\zeta'}^{2i-1}}}\!\!\! \g_\alpha\oplus
   \!\!\!  \bigoplus_{{\alpha\in \Delta'_\varrho\cap \Delta_{\zeta'}^{2i-1}}} \!\!\!(\g_\alpha\oplus  \g_{\varrho(\alpha)})\cap \mathfrak{k}^\C_\varrho. $$
     Given $\alpha\in \g_{2i-1}^{\zeta'}$, since $\alpha(\zeta')/\mathrm{i}$ is odd, we must have $\alpha(\zeta_j)\neq 0$ for some $j\in I_1$. Hence $\alpha(\zeta'-\tilde\zeta)/\mathrm{i}<\alpha(\zeta')/\mathrm{i} $ and we conclude that
  $$\g_{2i-1}^{\zeta'}\cap \mathfrak{k}_\varrho^\C  \subset \bigoplus_{0\leq j<2i-1}\g^{\zeta'-\zeta}_j.$$

 The statement of the theorem   follows now from Propositions \ref{norminf} and \ref{norm2}.
\end{proof}

Next we will study in detail the case  $G_3(\mathbb{R}^6)$. Take as base point  of  $G_3(\mathbb{R}^6)$ the $3$-dimensional real subspace $V_0\oplus V_1\oplus \overline{V}_1$, where $V_1$ is the one-dimensional isotropic subspace spanned by $\overline{u}_1$. This choice establishes the identification
$$G_3(\mathbb{R}^6)\cong SO(6)/S(O(3)\times O(3))$$ and the  corresponding involution is $\sigma_{\varrho,1}=\mathrm{Ad}(\exp \pi\zeta_1)\circ \sigma_\varrho$. Following our classification procedure, we also identify $G_3(\mathbb{R}^6)$ with $P_{\zeta_1}^{\sigma_\varrho}$ via the totally geodesic embedding \eqref{tttt}. From Theorem \ref{cangrass}, we have six classes of harmonic maps into $G_3(\mathbb{R}^6)$:
\begin{align*}
 &(\zeta_1,\sigma_\varrho),\,\,\,\,  (\zeta_1+\zeta_2,\sigma_\varrho),\,\,\,\, (\zeta_2, \sigma_{\varrho,1}),\,\,\,\,(\zeta_1, \sigma_{\varrho,2}), \,\,\,\, (\zeta_1+\zeta_2,\sigma_{\varrho,2}),\,\,\,\, (\zeta_2,\mathrm{Ad}(\exp\pi\zeta_2)\circ \sigma_{\varrho,1}).
\end{align*}

\begin{thm}\label{36} Let $\varphi:S^2\to G_3(\mathbb{R}^6)$ be an harmonic map.
  \begin{enumerate}
    \item If $\varphi$ is associated to the pair $(\zeta_1,\sigma_\varrho)$ then $\varphi$ is $S^1$-invariant and, up to isometry, is given by \begin{equation}\label{mixedpair}\varphi=V_0\oplus V\oplus \overline{V},\end{equation} where  $V$ is a holomorphic isotropic subbundle of $S^2\times V_0^\perp$ of rank $1$ satisfying $\partial V\subseteq \overline{V}^\perp$.
    \item  If $\varphi$ is associated to the pair $(\zeta_1+\zeta_2,\sigma_\varrho)$ and is $S^1$-invariant, then, up to isometry,
    \begin{equation}\label{varphi1}
    \varphi=V_0\oplus (W\cap V^\perp) \oplus (\overline {W\cap V^\perp }),\end{equation} where $V\subset W$ are holomorphic isotropic  subbundles of $S^2\times V_0^\perp$ of rank $1$ and $2$, respectively, satisfying $\partial V\subset W$ and $\partial W\subset \overline W^\perp.$
        \item If $\varphi$ is associated to the pair $(\zeta_2,\sigma_{\varrho,1})$ and is $S^1$-invariant, then, up to isometry,    \begin{equation}\label{varphi2}\varphi= \{(L_1\oplus \overline{L}_1)^\perp\cap (V_0\oplus V_1\oplus\overline{V}_1)\}\oplus(L_2\oplus \overline{L}_2),\end{equation} where $L_1$ and $L_2$ are holomorphic isotropic  bundle lines of $S^2\times (V_0\oplus V_1\oplus\overline{V}_1)$ and $S^2\times (V_0\oplus V_1\oplus\overline{V}_1)^\perp$, respectively. \end{enumerate}

          The corresponding extended solutions have uniton number $2$, $4$, and $2$, respectively, with respect to the standard representation of $SO(6)$.
            The harmonic maps in the classes  $(\zeta_1, \sigma_{\varrho,2})$, $(\zeta_1+\zeta_2,\sigma_{\varrho,2})$, and $(\zeta_2,\mathrm{Ad}(\exp\pi\zeta_2)\circ \sigma_{\varrho,1})$ are  precisely the orthogonal complements of the harmonic maps in the classes $(\zeta_1,\sigma_\varrho)$,  $(\zeta_1+\zeta_2,\sigma_\varrho)$, and   $(\zeta_2, \sigma_{\varrho,1})$, respectively.
            \end{thm}

\begin{proof}
 For the first two classes, and according to our classification procedure, we
identify $G_3(\mathbb{R}^6)$ with $P_{\zeta_1}^{\sigma_\varrho}$ via the totally geodesic embedding
$g\cdot (V_0\oplus V_1\oplus \overline{V}_1)\mapsto \exp(\pi\zeta_1)g\sigma_{\varrho,1}(g^{-1}).$ In these two cases, $T_{\sigma_\varrho}$-invariant extended solutions $\Phi$ associated to harmonic maps
 $\varphi=g\cdot (V_0\oplus V_1\oplus \overline{V}_1)$ satisfy
\begin{equation}\label{phimenos1}
\Phi_{-1}=\exp(\pi\zeta_1)g\sigma_{\varrho,1}(g^{-1})=\exp(\pi\zeta_1)(\pi_\varphi-\pi_{\varphi}^\perp)\exp(\pi\zeta_1)s_0.
\end{equation}

 First we consider the harmonic maps associated to the pair $(\zeta_1,\sigma_\varrho)$.  We have $r(\zeta_1)=1$ and
 $$(\mathfrak{u}^0_{\zeta_1})_{\sigma_{\varrho}}=(\mathfrak{p}^{\zeta_1}_{0})^\perp\cap \mathfrak{k}^\C_{{\varrho}}=\g_1^{\zeta_1}\cap \mathfrak{k}^\C_{{\varrho}}.$$
 Consequently any such harmonic map  is $S^1$-invariant. Write
 $\gamma_{\zeta_1}(\lambda)=\lambda^{-1}\pi_{V_1}+\pi^\perp_{V_1\oplus \overline{V}_1}+\lambda \pi_{\overline{V}_1},$
 where $V_1$ is the one-dimensional isotropic space spanned by $\overline{u}_1$.  Let $\Phi:S^2\setminus D\to U_{\zeta_1}^{\sigma_\varrho}$ be an extended solution associated to the harmonic map $\varphi$. Then, by $S^1$-invariance, we can write
 \begin{equation}\label{phicase1}
 \Phi_\lambda=\lambda^{-1}\pi_V+ \pi^\perp_{V\oplus \overline{V}}+\lambda\pi_{\overline V},
  \end{equation} where $V$ is a holomorphic isotropic subbundle of $S^2\times \mathbb{R}^{6}$ of rank $1$ satisfying $\partial V\subseteq \overline{V}^\perp$. The $T_{\sigma_\varrho}$-invariance of $\Phi$ implies that $V_0\subset (V\oplus\overline{V})^\perp$.
Equating \eqref{phimenos1} and \eqref{phicase1}, we get, up to isometry, $\varphi=V_0\oplus V\oplus \overline{V}$.

For the case  $(\zeta_1+\zeta_2,\sigma_{\varrho})$, since
\begin{equation}\label{zeta1+zeta2}
\gamma_{\zeta_1+\zeta_2}(\lambda)=\lambda^{-2}\pi_{V_1}+\lambda^{-1}\pi_{V_2\cap V_1^\perp}+\pi_{V_2\oplus \overline{V}_2}^\perp+\lambda\pi_{\overline{V}_2\cap \overline{V}_1^\perp}+\lambda^2\pi_{\overline{V}_1},
\end{equation}
 any $S^1$-invariant harmonic map $\varphi$ in this class  admits an extended solution of the form
\begin{equation}\label{phicase2}
\Phi_\lambda= \lambda^{-2}\pi_{V}+\lambda^{-1}\pi_{W\cap V^\perp}+\pi_{W\oplus \overline{W}}^\perp+\lambda\pi_{\overline{W}\cap \overline{V}^\perp}+\lambda^2\pi_{\overline{V}},
\end{equation}
where  $V\subset W$ are holomorphic isotropic  subbundles  of rank $1$ and $2$, respectively, satisfying $\partial V\subset W$ and $\partial W\subset \overline W^\perp.$
By $T_{\sigma_\varrho}$-invariance, we must have $V_0\subset (W\oplus\overline{W})^\perp$, hence $V\subset W$ are subbundles of $S^2\times V_0^\perp$. Equating \eqref{phimenos1} and \eqref{phicase2}, we get \eqref{varphi1}.

For the case  $(\zeta_2,\sigma_{\varrho,1})$, we identify $G_3(\mathbb{R}^6)$ with $P_{\zeta_2}^{\sigma_{\varrho,1}}=\exp\pi\zeta_1P_{\zeta_2-\zeta_1}^{\sigma_\varrho}$ via the totally geodesic embedding
\begin{equation}\label{oao}
g\cdot (V_0\oplus V_1\oplus \overline{V}_1)\mapsto g\sigma_{\varrho,1}(g^{-1}).\end{equation}
Extended solutions $\Phi$ associated to $S^1$-invariant harmonic maps in this class must be of the form
\begin{equation}\label{phiws}
\Phi_\lambda=\lambda^{-1}\pi_W+ \pi_{W\oplus \overline W}^\perp+\lambda\pi_W,\end{equation} where  $W$ is a holomorphic isotropic  subbundle  of rank $2$.
By $T_{\sigma_{\varrho,1}}$-invariance, we must have
$[\pi_W,\pi_{V_0\oplus  V_1\oplus\overline{V}_1}]=0,$
which means that $W$ must be of the form $W=L_1\oplus L_2$, where  $L_1$ and $L_2$, respectively, are holomorphic isotropic  bundle lines of $S^2\times (V_0\oplus V_1\oplus\overline{V}_1)$ and $S^2\times (V_0\oplus V_1\oplus\overline{V}_1)^\perp$.

On the other hand, in view of \eqref{oao}, we have
$\Phi_{-1}=(\pi_\varphi-\pi_\varphi^\perp)\exp(\pi\zeta_1) s_0.$
Equating this with \eqref{phiws},  we conclude that
\eqref{varphi2} holds.
The remaining cases are treated similarly.
\end{proof}
\begin{rem}
  The first two classes of $S^1$-invariant harmonic maps $\varphi:S^2\to G_3(\mathbb{R}^6)$ in Theorem \ref{36} factor through $G_2(\mathbb{R}^5)$. That is, for any such harmonic map $\varphi$, there exists $\tilde{\varphi}:S^2\to G_2(\mathbb{R}^5)$, where we identify $\mathbb{R}^5$ with $V_0^\perp$, such that $\varphi=V_0\oplus \tilde{\varphi}$. An explicit construction of all harmonic maps from $S^2$ into  $G_2(\mathbb{R}^n)$ can be found in \cite{woodG2}. In that paper, harmonic maps of the form \eqref{mixedpair} are called \emph{real mixed pairs}. We emphasise that the harmonic maps into $G_3(\mathbb{R}^6)$  associated to extended solutions in the corresponding unstable manifolds need not to factor through $G_2(\mathbb{R}^5)$ in the same way.
 \end{rem}

  Let us consider  the case $(\zeta_1+\zeta_2,\sigma_{\varrho})$.
Taking into account the Weierstrass representation of Theorem \ref{sigmaweirstrass}, any extended solution $\Phi:S^2\setminus D\to U^{\sigma_{\varrho}}_{\zeta}(SO(6))$, with $\zeta=\zeta_1+\zeta_2$, can be written as
$\Phi=\exp{C}\cdot \gamma_{\zeta}$, for some meromorphic vector-valued function
$C:S^2\to (\mathfrak{u}^0_\zeta)_{\sigma_\varrho}$. We have $r(\zeta)=3$ and
$$ (\mathfrak{u}^0_{\zeta})_{\sigma_\varrho}=(\g_1^\zeta\oplus\g_2^\zeta\oplus \g_3^\zeta)\cap \mathfrak{k}^\C_{\varrho}\oplus \lambda (\g_2^\zeta\oplus \g_3^\zeta)\cap \mathfrak{m}^\C_{\varrho}\oplus \lambda^2 \g_3^\zeta\cap  \mathfrak{k}^\C_{\varrho}.$$
Moreover,
\begin{eqnarray*}
&\g_1^\zeta\cap \mathfrak{k}^\C_{\varrho}=\g_{L_1-L_2}\oplus \{(\g_{L_2- L_3}\oplus\g_{L_2+L_3}) \cap \mathfrak{k}^\C_{\varrho}\},\quad
\g_2^\zeta\cap \mathfrak{k}^\C_{\varrho}=(\g_{L_1+ L_3}\oplus\g_{L_1-L_3} )\cap \mathfrak{k}^\C_{\varrho},\\
& \g_3^\zeta\cap \mathfrak{k}^\C_{\varrho}=\g_{L_1+L_2},\quad
(\g_2^\zeta\oplus \g_3^\zeta)\cap \mathfrak{m}^\C_{\varrho}=\g_2^\zeta\cap \mathfrak{m}^\C_{\varrho}=(\g_{L_1- L_3}\oplus \g_{L_1+ L_3})\cap \mathfrak{m}^\C_{\varrho}.
\end{eqnarray*}

%
% Next we  fix, for each positive root $\alpha\in \Delta^+$, a suitable vector root space $X^\alpha\in\g_\alpha$ as follows: choose arbitrary  vector root spaces $X^{L_1-L_2}
% $ and  $X^{L_2-L_3}$; and set
% \begin{align*}
%  X^{L_1-L_3}&:=[X^{L_1-L_2},X^{L_2-L_3}],\,\,\,\,X^{L_2+L_3}:=\sigma_\varrho (X^{L_2-L_3}),  \\ X^{L_1+L_2}&:=[X^{L_2+L_3},X^{L_1-L_3}],\,\,\,\,\,X^{L_1+L_3}:=\sigma_\varrho (X^{L_1-L_3}).
% \end{align*}

Write
\begin{equation}\label{CC}
C=C_0+\lambda C_1+\lambda^2 C_2,\quad C_0=c_0^1+c_0^2+c_0^3,\quad C_1=c_1^2+c_1^3,\quad C_2=c_2^3\end{equation}
where the functions $c_0^i:S^2\to \g_i^\zeta\cap \mathfrak{k}^\C_{\varrho}$,  $c_1^i:S^2\to \g_i^\zeta\cap \mathfrak{m}^\C_{\varrho}$, and
$c_2^3:S^2\to \g_3^\zeta\cap \mathfrak{k}^\C_{\varrho}$ are meromorphic functions.
Clearly, $c_1^3=0$.
Consider the root vectors defined by \eqref{geradores}. Since $\sigma_\varrho(X_{2,3})=-Y_{2,3}$ and  $\sigma_\varrho(X_{1,3})=-Y_{1,3}$, we can write
\begin{align*}
  c_0^1=aX_{1,2}+b(X_{2,3}-Y_{2,3}),\,\,c_0^2=c(X_{1,3}-Y_{1,3}),\,\, c_0^3=dY_{1,2},\,\,
  c_1^2=e(X_{1,3}+Y_{1,3}),\,\,c_2^3=fX_{1,2}
\end{align*}
  in terms of $\C$-valued meromorphic functions $a$, $b$, $c$, $d$, $e$,  $f$.

Taking into account the results of Section \ref{weircond}, $\Phi=\exp{C}\cdot \gamma_{\zeta}$ is an
 extended solution if and only if, in the expression
$$(\exp C)^{-1}(\exp C)_z=C_z-\frac{1}{2!}(\mathrm{ad} C)C_z+\frac{1}{3!}(\mathrm{ad} C)^2C_z,$$
we have:
\begin{enumerate}
  \item[a)] the independent coefficient should have zero component in each $\g_{2}^{\zeta}$ and $\g_{3}^{\zeta},$ that is
  \begin{align}\label{o}
    {c_0^2}_z-\frac{1}{2}[c_0^1,{c_0^1}_z]=0,\,\,\,\, {c_0^3}_z-\frac{1}{2}[c_0^1,{c_0^2}_z]-\frac{1}{2}[c_0^2,{c_0^1}_z]+\frac16[c_0^1,[c_0^1,{c_0^1}_z]]=0;
  \end{align}
   \item[b)] the $\lambda$ coefficient should have zero component in  $\g_{3}^{\zeta},$ that is
   \begin{equation}\label{1}
     [c_0^1,{c_1^2}_z]+[c_1^2,{c_0^1}_z]=0.
   \end{equation}
\end{enumerate}
From equations \eqref{o} we get the equations (prime denotes $z$-derivative)
\begin{equation}\label{merocond}
2c'=ab'-ba',\,\,\,\,\, 3d'=3 cb'-bc';
\end{equation}
on the other hand, observe that \eqref{1} always holds since $$[c_0^1,{c_1^2}_z]+[c_1^2,{c_0^1}_z]\in [\g_1^\zeta\cap \mathfrak{k}^\C_{\varrho},\g_2^\zeta\cap \mathfrak{m}^\C_{\varrho}]\subset \g_3^\zeta\cap \mathfrak{m}^\C_{\varrho}=\{0\}.$$

Hence we conclude that, any extended solution $\Phi:S^2\setminus D\to U^{\sigma_{\varrho}}_{\zeta}(SO(6))$, with $\zeta=\zeta_1+\zeta_2$, of the form
$\Phi=\exp{C}\cdot \gamma_{\zeta}$, can be constructed as follows: choose arbitrary meromorphic functions $a$, $b$, $e$  and $f$; integrate equations \eqref{merocond} to obtain the meromorphic functions $c$ and $d$; $C$ is then given by \eqref{CC}.

\begin{eg}
 Choose $a(z)=b(z)=z$. From \eqref{merocond}, we can take $c(z)=1$ and $d(z)=z$. This data defines the matrix $C_0$ and the $S^1$-invariant extended solution $\Phi^0=\exp C_0\cdot \gamma_{\zeta}$, where  the loop $\gamma_\zeta$, with $\zeta=\zeta_1+\zeta_2$, is given by \eqref{zeta1+zeta2}.
The extended solutions $\Phi:S^2\to U_{\zeta}^{\sigma_\varrho}(SO(6))$ satisfying $\Phi^0=u_{\zeta}\circ\Phi$ are of the form $\Phi=\exp C\cdot \gamma_\zeta$, where the matrix $C=C_0+C_1\lambda+C_2\lambda^2$ is given by
 $$C= \left(
          \begin{array}{cccccc}
            0 & z & 1 & 0 & z & -1 \\
            0 & 0 & z & -z & 0 & -z \\
            0 & 0 & 0 & 1 & z & 0 \\
            0 & 0 & 0 & 0 & 0 & 0 \\
            0 & 0 & 0 & -z & 0 & 0 \\
            0 & 0 & 0 & -1 & -z & 0 \\
          \end{array}
        \right)+ \left(
          \begin{array}{cccccc}
            0 & 0 & e\lambda & 0 & f\lambda^2 & -e\lambda \\
            0 & 0 & 0 & -f\lambda^2 & 0 & 0 \\
            0 & 0 & 0 & e\lambda & 0 & 0 \\
            0 & 0 & 0 & 0 & 0 & 0 \\
            0 & 0 & 0 & 0 & 0 & 0 \\
            0 & 0 & 0 & -e\lambda & 0 & 0 \\
          \end{array}
        \right),$$
 with respect to the complex orthonormal basis $\mathbf{u}=\{u_1,u_2,u_3,\overline{u}_1,\overline{u}_2,\overline{u}_3\}$, where $e$ and $f$ are arbitrary meromorphic functions on $S^2$.
The
 holomorphic vector bundles $V$ and $W$ associated to the $S^1$-invariant extended solution $\exp C_0\cdot \gamma_{\zeta}$ are given by $V=\exp{C_0}\cdot {V}_1$ and $W=\exp{C_0}\cdot V_2$, and we have, with respect to  the basis $\mathbf{u}$,
\begin{align*}
  V&=\mathrm{span}\{(12-12z^2-z^4,-4z^3,12-6z^2,12,-12z,-12+6z^2 ) \}\\
  W&=\mathrm{span}\{(6z+z^3,3z^2,3z,0,3,-3z ) \}\oplus V.
\end{align*}

\end{eg}

\subsection{Outer symmetric $SU(2n+1)$-spaces.}\label{su(2n+1)-outer}
%\subsubsection{The Lie algebra $\mathfrak{su}(m)$.}

Let $E_j$ be the square $(m\times m)$-matrix whose $(j,j)$-entry is $\mathrm{i}$ and all other entries are $0$.
The complexification $\lt^\C$ of the algebra $\lt$ of diagonal matrices $\sum a_iE_i$, with $a_i\in \mathbb{C}$ and $\sum a_i=0$,   is a Cartan subalgebra of $\mathfrak{su}(m)^\C$. Let $\{L_1,\ldots,L_{m}\}$ be the dual basis of $\{E_1,\ldots,E_{m}\}$, that is  $L_i(E_j)=\mathrm{i}\delta_{ij}$. The roots of $\mathfrak{su}(m)$ are the vectors $ L_i- L_j$, with $i\neq j$ and $1\leq i,j\leq m-1$ and $\Delta^+=\{L_i-L_j\}_{i<j}$ is a positive root system with positive simple roots $\alpha_i=L_i-L_{i+1}$, for $1\leq i\leq m-1$.
For $i\neq j$, the matrix $X_{i,j}$ whose $(i,j)$ entry is $1$ and all other entries are $0$ generate the root space $\g_{L_i-L_j}$.
The dual basis of $\Delta_0=\{\alpha_1,\ldots,\alpha_{m-1}\}$
in $\mathrm{i}\mathfrak{t}^*$ is formed by the matrices
\begin{align*}
H_i=\frac{m-i}{m}(E_1+\ldots+ E_i)- \frac{i}{m}(E_{i+1}+\ldots+ E_{m}).
\end{align*}

\subsubsection{Special Lagrangian spaces}

Consider on $\mathbb{R}^{2m}$ the standard inner product $\langle\cdot,\cdot\rangle$ and the canonical orthonormal basis $\mathbf{e}^{2m}=\{e_1,\ldots, e_{2m}\}.$
Define the orthogonal complex structure $I$  by $I(e_i)=e_{2m+1-i}$, for $i\in \{1,\ldots,m\}$. A \emph{Lagrangian subspace} of $\mathbb{R}^{2m}$ (with respect to $I$) is a $m$-dimensional subspace $L$ such that $IL\perp L$.  Let $\mathcal{L}_m$ be the space of all
Lagrangian subspaces of  $\mathbb{R}^{2m}$ and $L_0\in\mathcal{L}_m$ the Lagrangian subspace generated by $\mathbf{e}^{m}=\{e_1,\ldots,e_m\}$. The unitary group $U(m)$ acts transitively on  $\mathcal{L}_m$, with isotropy group at $L_0$ equal to $SO(m)$, and  $\mathcal{L}_m$
 is a reducible symmetric space that can be identified with $U(m)/SO(m)$ (see \cite{Zi} for details).

  The space $\mathcal{L}_m$ can also be interpreted as the set of all orthogonal linear maps $\tau:\mathbb{R}^{2m}\to\mathbb{R}^{2m}$ satisfying $\tau^2=e$ and $I\tau=-\tau I$. Indeed, if $V_\pm$ are the $\pm 1$ eigenspaces of $\tau$, then $IV_+=V_-$ and $IV_+\perp V_+$, that is $V_+$ is Lagrangian. From this point of view, $U(m)$ acts on $\mathcal{L}_m$ by conjugation: $g\cdot \tau=g\tau g^{-1}$. Let $\tau_0\in \mathcal{L}_m$ be the orthogonal linear map corresponding to $L_0$, that is, ${\tau_0}_{|_{L_0}}=e$ and ${\tau_0}_{|_{IL_0}}=-e$.
The corresponding involution on $U(m)$ is given by $\sigma(g)=\tau_0g\tau_0$ and the Cartan embedding $\iota:\mathcal{L}_m\hookrightarrow U(m)$ is given by $\iota(\tau)=\tau \tau_0$.

%Observe that if we interpret $U(m)$ as the the matrices in $SO(2m)$ which commute with the complex structure $J$
%\begin{equation}\label{mergulho}
%A+B\mathrm{i}\mapsto \left(
%                              \begin{array}{cc}
%                                A & -B \\
%                                B & A \\
%                              \end{array}
%                            \right),\end{equation}
%                            the complex conjugation corresponds to the Cartan involution of $G_m(\mathbb{R}^{2m})$.

The totally geodesic submanifold $\mathcal{L}^s_m:=SU(m)/SO(m)$ of  $U(m)/SO(m)$ is also known as the \emph{space of special Lagrangian subspaces of $\mathbb{R}^{2m}$}. It is an irreducible outer symmetric $SU(m)$-space.

\subsubsection{Harmonic maps into $\mathcal{L}^s_{2n+1}$}

Take $m=2n+1$. The non-trivial involution $\varrho$ of the Dynkin diagram of $\mathfrak{su}(2n+1)^\C$ is given by $\varrho(\alpha_i)=\alpha_{2n+1-i}$. In particular, $\varrho$ does not fix any root in $\Delta_0$ and there exists only one class of outer symmetric $SU(2n+1)$-spaces. The semi-fundamental basis $\pi_{\mathfrak{k}_\varrho}(\Delta_0)=\{\beta_1,\ldots,\beta_{n}\}$ is given by
$\beta_i=\frac12(\alpha_i+\alpha_{2n+1-i})$ whereas the dual basis $\{\zeta_1,\ldots,\zeta_{n}\}$ is given by
$$\zeta_i=H_i+H_{2n+1-i}=E_1+\ldots+E_i-(E_{2n+2-i}+\ldots +E_{2n+1}),$$
for $1\leq i\leq n$. Since each $\zeta_i$ belongs to the integer lattice $\mathfrak{I}(SU(2n+1))$,
the $\varrho$-semi-canonical elements of $SU(2n+1)$ are precisely the elements $\zeta=\sum_{i=1}^{n}m_i\zeta_i$ with $m_i\in\{0,1,2\}$.

Let $\mathbf{e}^{2n+1}=\{e_1,\ldots,e_{2n+1}\}$ be the canonical orthonormal basis of $\mathbb{R}^{2n+1}$. Identify $\C^{2n+1}$ with $(\mathbb{R}^{4n+2},I)$, where $I$ is defined as above.
Set
$$v_j=\frac{1}{\sqrt2}(e_j+\mathrm{i}e_{2n+2-j}),$$
for $1\leq j\leq n$, $v_{n+1}=e_{n+1}$ and $v_{2n+2-j}=\overline{v}_j$. Take the matrices $E_j$ with respect to the complex basis $\mathbf{v}=\{v_1,\ldots,v_{2n+1}\}$ of $\C^{2n+1}$.
Hence $\tau_0E_j\tau_0=-E_{2n+2-j}$ and
the fundamental involution $\sigma_\varrho$ is given by $\sigma_\varrho(g)=\tau_0g\tau_0$.
The fundamental outer symmetric $SU(2n+1)$-space is the space of special Lagrangian subspaces $\mathcal{L}_{2n+1}^s=SU(2n+1)/SO(2n+1)$, and this is the unique  outer symmetric $SU(2n+1)$-space.

%
%\begin{thm}
%  Each harmonic map $\varphi:S^2\to \mathcal{L}_{2n+1}^s$ belongs to one of the following classes: $(\sum_{i\in I}\zeta_i, \sigma_\varrho)$, with $I\subseteq \{1,\ldots,n\}$.
%\end{thm}
\vspace{.20in}

Next we consider in detail harmonic maps into $\mathcal{L}_{3}^s$. In this case we have two non-zero $\varrho$-semi-canonical elements, $\zeta_1$ and $2\zeta_1$, and consequently two classes of harmonic maps, $(\zeta_1,\sigma_\varrho)$ and $(\zeta_1,\sigma_{\varrho,1})$.
Since $\zeta_1=E_1-E_3$, we have
$r(\zeta_1)=(L_1-L_3)(\zeta_1)/\mathrm{i}=2$. Let $W_1$, $W_2$ and $W_3$ be the complex one-dimensional images of $E_1$, $E_2$ and $E_3$, respectively. Any extended solution $$\Phi:S^2\setminus D\to U_{\zeta_1}^{\sigma_\varrho}(SU(2n+1))$$ is given by $\Phi=\exp C\cdot \gamma_{\zeta_1}$, with $\gamma_{\zeta_1}(\lambda)=\lambda^{-1}\pi_{W_3}+\pi_{W_2}+\lambda \pi_{W_1}$, for some holomorphic vector-valued function
$C:S^2\setminus D\to (\mathfrak{u}^0_{\zeta_1})_{\sigma_\varrho}$, where
$$(\mathfrak{u}^0_{\zeta_1})_{\sigma_\varrho}=(\mathfrak{p}_0^{\zeta_1})^\perp\cap\mathfrak{k}^\C_{\varrho}+\lambda(\mathfrak{p}_1^{\zeta_1})^\perp\cap
\mathfrak{m}^\C_{\varrho} $$
and
\begin{align*}
 (\mathfrak{p}_0^{\zeta_1})^\perp\cap\mathfrak{k}^\C_{\varrho}=(\g_{L_1-L_2}\oplus\g_{L_2-L_3}\oplus \g_{L_1-L_3})\cap\mathfrak{k}^\C_{\varrho},\quad (\mathfrak{p}_1^{\zeta_1})^\perp\cap\mathfrak{m}^\C_{\varrho}= \g_{L_1-L_3}\cap\mathfrak{m}^\C_{\varrho}.
\end{align*}

Let $X_{i,j}$ be the square matrix whose $(i,j)$ entry is $1$ and all the other entries are $0$, with respect to the basis $\mathbf{v}$. The root space $\mathfrak{g}_{L_i-L_j}$ is spanned by $X_{i,j}$. We have
$\sigma_{\varrho}(X_{1,2})=-X_{2,3}$ and $\sigma_\varrho(X_{1,3})= -X_{1,3}$ (consequently, $\mathfrak{g}_{L_1-L_3}\subset \mathfrak{m}_{\varrho}^\C$).    Hence we can write $C=C_0+C_1\lambda$, with
$C_0=a(X_{1,2}-X_{2,3})$ and  $C_1=bX_{1,3}$,
for some meromorphic functions $a,b$ on $S^2$. The harmonicity equations do not impose any condition on these meromorphic functions, hence any harmonic map
$\varphi:S^2\to \mathcal{L}_{3}^s$ in the class $(\zeta_1,\sigma_\varrho)$ admits an extended solution of the form
\begin{equation}\label{su1}
\Phi=\exp\left(
            \begin{array}{ccc}
              0 & a & b\lambda \\
              0 & 0 & -a \\
              0 & 0 & 0 \\
            \end{array}
          \right)\cdot \gamma_{\zeta_1}=\left(
                                          \begin{array}{ccc}
                                            1 & a & \frac12(-a^2+2b\lambda) \\
                                            0 & 1 & -a \\
                                            0 & 0 & 1 \\
                                          \end{array}
                                        \right)\cdot \gamma_{\zeta_1},\end{equation}
                                        and $\varphi$ is recovered by setting $\varphi=\Phi_{-1}\tau_0$. Similarly, one can see that the class of harmonic maps in $(\zeta_1,\sigma_{\varrho,1})$
                                    admits an extended solution of the form
\begin{equation}\label{su2}\Phi=\left(
                                          \begin{array}{ccc}
                                            1 & a & \frac12(a^2+2b\lambda) \\
                                            0 & 1 & a \\
                                            0 & 0 & 1 \\
                                          \end{array}
                                        \right)\cdot \gamma_{\zeta_1},\end{equation}
                                        with no restrictions on the meromorphic functions $a$ and $b$.

H. Ma established (cf. Theorem 4.1 of \cite{Ma}) that harmonic maps $\varphi:S^2\to \mathcal{L}_{3}^s$ are essentially of two types: 1) $\iota_\sigma\circ \varphi$ is a \emph{Grassmannian solution} obtained from a full harmonic  map $f:S^2\to \mathbb{R}P^2\subset \mathbb{C}P^2$, where $\iota_\sigma$ is the Cartan embedding of $ \mathcal{L}_{3}^s$ in $SU(3)$; 2) up to left multiplication by a constant, $\iota_\sigma\circ \varphi$ is of the form $(\pi_{\beta_1}- \pi_{\beta_1}^\perp) (\pi_{\beta_2}- \pi_{\beta_2}^\perp)$, where $\beta_1$ is a \emph{Frenet pair} associated to a full totally istotropic holomorphic map $g:S^2\to\mathbb{C}P^2$ and $\beta_2$ is a rank $1$ holomorphic subbundle of $G'(g)^\perp$, where $G'(g)$ is the first \emph{Gauss bundle} of $g$. Observe that if, in the second case,  $\beta_2$ coincides with $g$, then $\iota_\sigma\circ\varphi$ is a {Grassmannian solution} obtained from the full harmonic  map $f:=G'(g)$ from $S^2$ to $\mathbb{R}P^2$, that is,  $\varphi$ is of type 1). Comparing this with our description, it is not difficult to see that harmonic maps of type 1) are $S^1$-invariant extended solutions (take $b=0$ in \eqref{su1} and \eqref{su2}) and   harmonic maps of type 2) are associated to extended solutions with values in the corresponding unstable manifolds (which corresponds to an arbitrary choice of $b$ in \eqref{su1} and \eqref{su2}). H. Ma also established a  purely algebraic explicit construction of such harmonic maps in terms of meromorphic data on $S^2$, which is consistent with our results.

%Under the natural embedding  $SU(3)\hookrightarrow SO(6)$, harmonic maps into $SU(3)/SO(3)$ belong to the classes $(\zeta_2,\sigma_{\varrho,1})$ and $(\zeta_2,\mathrm{Ad}(\exp\pi\zeta_2)\circ\sigma_{\varrho,1})$ of harmonic maps into $G_3(\mathbb{R}^6)$. This can be seen as follows. Fix on $\mathbb{R}^6$ the canonical basis $\mathbf{e}^6=\{e_1,e_2,e_3,e_4,e_5,e_6\}$ and consider the complex structure $I$ defined as above: $Ie_1=e_6$, $Ie_2=e_5$ and $Ie_3=e_4$.
%Identify $\C^3\cong(\mathbb{R}^6,I)$.
%Define the complex
%basis $\{u_1,u_2,u_3,\overline{u}_1,\overline{u}_2,\overline{u}_3\}$ of $\mathbb{C}^6=(\mathbb{R}^6)^\C$ by
%$$u_1=\frac{1}{\sqrt 2}(e_1+\mathrm{i}e_3),\,\, u_2=\frac{1}{\sqrt 2}(Ie_1+\mathrm{i}Ie_3),\,\,\, u_3=\frac{1}{\sqrt 2}(Ie_2+\mathrm{i}e_2),$$
%which satisfies (\ref{us}).
%With respect to these choices, the image of the $\varrho$-canonical element $\zeta_1$ of $\mathfrak{su}(3)$ under the natural embedding $SU(3)\hookrightarrow SO(6)$  is precisely the $\varrho$-canonical element $\zeta_2$ of $\mathfrak{so}(6)$ given by (\ref{zetaii}).

\subsection{Outer symmetric $SU(2n)$-spaces.}With the same notations of Section \ref{su(2n+1)-outer}, the non-trivial involution $\varrho$ of the Dynkin diagram of $\mathfrak{su}(2n)$ is given by $\varrho(\alpha_i)=\alpha_{2n-i}$, and $\varrho$ fixes the root $\alpha_n$. The semi-fundamental basis $\pi_{\mathfrak{k}_\varrho}(\Delta_0)=\{\beta_1,\ldots,\beta_{n-1}\}$ is given by $\beta_1=\alpha_n$ and $\beta_i=\frac12(\alpha_i+\alpha_{2n-i})$ if $i\geq 2$; whereas its dual basis $\{\zeta_1,\ldots,\zeta_{n-1}\}$ is given by
\begin{align*}
  \zeta_1&=H_n=\frac12({E_1}+\ldots+E_n)-\frac12(E_{n+1}+\ldots+ E_{2n})\\
  \zeta_i&= H_{i-1}+H_{2n-i+1}=E_1+\ldots+E_{i-1}-(E_{2n+2-i}+\ldots+E_{2n}),\,\,\,\,\mbox{for $2\leq i\leq n-1$}.
\end{align*}

By Theorem \ref{murak}, there exist two conjugacy classes of outer involutions: the fundamental outer involution $\sigma_\varrho$ and $\sigma_{\varrho,1}$. These outer involutions correspond to the symmetric spaces $SU(2n)/Sp(n)$ and $SU(2n)/SO(2n)$, respectively.  Observe that $\zeta_1$ does not belong to the integer lattice $\mathfrak{I}'(SU(2n)^{\sigma_\varrho})$ since $\exp 2\pi\zeta_1=-e$.

\subsubsection{Harmonic maps into the space of special unitary quaternionic structures on $\mathbb{C}^{2n}$. }

 A \emph{unitary quaterninonic structure} on the standard hermitian space $(\mathbb{C}^{2n},\langle\cdot,\cdot\rangle )$ is a conjugate linear map $J:\C^{2n}\to \C^{2n}$ satisfying $J^2=-Id$ and
 $\langle v,w\rangle= \langle J\,w,J\,v\rangle$ for all $v,w\in \C^{2n}$. Consider as base point the quaternionic structure $J_o$ defined by $J_o(e_i)=e_{2n+1-i}$ for each $1\leq i\leq n$, where
 $\mathbf{e}^{2n}=\{e_1,\ldots,e_{2n}\}$
 is the canonical hermitian basis of $\mathbb{C}^{2n}$. The unitary group $U(2n)$ acts transitively on the space  of unitary quaternionic structures on $\C^{2n}$ with isotropy group at $J_o$ equal to  $Sp(n)$, and thus $M=U(2n)/Sp(n)$. This is a reducible symmetric space with involution $\sigma:U(2n)\to U(2n)$ given by $\sigma(X)=J_oXJ_o^{-1}$, but the totally geodesic submanifold $\mathcal{Q}^s_n:=SU(2n)/Sp(n)$ is an irreducible symmetric space, which we call the space of \emph{special unitary quaternionic structures} on   $\mathbb{C}^{2n}$ (see \cite{Zi} for details).
 If we consider the matrices $E_i$ with respect to the complex basis $\mathbf{v}=\{v_1,\ldots,v_{2n}\}$ defined by
 \begin{equation}\label{ves}
 v_j=\frac{1}{\sqrt2}(e_j+\mathrm{i}e_{2n+1-j}),
 \end{equation}
 for $1\leq j\leq n$, and $v_{2n+1-j}=\overline{v}_j$, we see that $J_oE_jJ_o^{-1}=-E_{2n+1-j}$, and consequently
  we have $\sigma=\sigma_\varrho$.

 \vspace{.10in}

 Next we consider with detail harmonic maps into $\mathcal{Q}^s_2$.
\begin{prop}
  Each harmonic map $\varphi:S^2\to\mathcal{Q}^s_2$ belongs to one of the following classes: $(2\zeta_1,\sigma_\varrho)$, and $(\zeta_2,\sigma_{\varrho,2})$.
\end{prop}
  \begin{proof}We start by identifying $\mathcal{Q}^s_2$ with $P_e^{\sigma_\varrho}$.

  The $\varrho$-semi-canonical elements of $SU(4)$ are precisely the elements
$$2\zeta_1,\,4\zeta_1,\,\zeta_2,\,2\zeta_2,\,2\zeta_1+\zeta_2,\,2\zeta_1+2\zeta_2,\,4\zeta_1+\zeta_2,\,4\zeta_1+2\zeta_2.$$
 By Proposition \ref{fundcan}, all these elements correspond to  the symmetric space $\mathcal{Q}^s_2$.

 We claim that $\exp\pi \zeta_2$ is not in the connected component $$P_e^{\sigma_\varrho}=\{gJ_og^{-1}J_o^{-1}|\,g\in SU(4)\}.$$
 In fact, $\exp(\pi \zeta_2)J_o=gJ_og^{-1}\cong gSp(n)$ for the unitary transformation $g$ defined by $g(e_1)=e_4$, $g(e_4)=e_1$, $g(e_2)=e_3$ and $g(e_3)=-e_2$. Since $\det g\neq 1$ we conclude that $\exp\pi \zeta_2$ does not belong to $P_e^{\sigma_\varrho}$. Similarly, one can check that $\exp\pi (2\zeta_1+\zeta_2)$ is not in $P_e^{\sigma_\varrho}$.

 Hence, since $\exp\pi2\zeta_1$ belongs to the centre of $SU(4)$, any harmonic map $\varphi:S^2\to\mathcal{Q}^s_2\cong P_e^{\sigma_\varrho} $ belongs to one of the following classes:
 $(2\zeta_1,\sigma_\varrho)$, $(\zeta_2,\sigma_{\varrho,2})$, and $(2\zeta_1+\zeta_2,\sigma_{\varrho,2})$.  It remains to check that, in view of Proposition \ref{norm2}, harmonic maps in the class
 $(2\zeta_1+\zeta_2,\sigma_{\varrho,2})$ can be normalized to harmonic maps in the class $(\zeta_2,\sigma_{\varrho,2})$.

It is clear that $2\zeta_1+\zeta_2\preceq \zeta_2$. On the other hand,  for any positive root $L_i-L_j\in \Delta^+$, with $i<j$, we have $(L_i-L_j)(2\zeta_1)/\mathrm{i}\leq (L_i-L_j)({2\zeta_1+\zeta_2})/\mathrm{i},$ where the equality holds in just one case: $(L_2-L_3)(2\zeta_1)= (L_2-L_3)(2\zeta_1+\zeta_2)=2\mathrm{i}.$ However, $\mathfrak{g}_{L_2-L_3}\subset \mathfrak{k}_{\sigma_{\varrho,2}}$, which means that the conditions of Proposition \ref{norm2} hold for $\zeta=2\zeta_1+\zeta_2$ and $\zeta'=\zeta_2$, and consequently harmonic maps in the class
 $(2\zeta_1+\zeta_2,\sigma_{\varrho,2})$ can be normalized to harmonic maps in the class $(\zeta_2,\sigma_{\varrho,2})$.
  \end{proof}

 Following the same procedure as before, one can see that any harmonic map $\varphi\to\mathcal{Q}^s_2$ in the class $(2\zeta_1,\sigma_\varrho)$ admits an extended solution of the form
 $$\Phi=\left(
          \begin{array}{cccc}
            1 & 0 & c_1+a\lambda & c_2 \\
            0 & 1 & c_3 & c_1-a\lambda \\
            0 & 0 & 1 & 0 \\
            0 & 0 & 0 & 1 \\
          \end{array}
        \right)\cdot \gamma_{2\zeta_1},$$
        where $c_1,c_2,c_3\in\C$ are constants, $a$ is a meromorphic function on $S^2$.  The harmonic map  is recovered by setting $\varphi=\Phi_{-1}J_o$. Reciprocally, given arbitrary complex constants $c_1,c_2,c_3$ and a meromorphic function $a:S^2\to \C$,
         such $\Phi$ is an extended solution associated to some harmonic map in the class $(2\zeta_1,\sigma_\varrho)$ (the harmonicity equations do not impose any restriction to $a$).

 Similarly,  any harmonic map $\varphi\to\mathcal{Q}^s_2$ in the class $(\zeta_2,\sigma_{\varrho,2})$ admits an extended solution of the form
 $$\Phi=\left(
          \begin{array}{cccc}
            1 & b & a & c \\
            0 & 1 & 0 & a \\
            0 & 0 & 1 & -b \\
            0 & 0 & 0 & 1 \\
          \end{array}
        \right)\cdot \gamma_{\zeta_2},$$
  where $a$, $b$ and $c$ are meromorphic functions satisfying $c'=ba'-b'a$.
Since $P_{\zeta_2}^{\sigma_{\varrho,2}}=\exp(\pi\zeta_2)P_e^{\sigma_\varrho}$, the harmonic map  is recovered by setting $\varphi=\exp(\pi\zeta_2)\Phi_{-1}J_o$.

\subsubsection{Harmonic maps into $\mathcal{L}^s_{2n}$. }
The outer symmetric $SU(2n)$-space that corresponds to the involution $\sigma_{\varrho,1}$ is the space of special Lagrangian subspaces  $\mathcal{L}^s_{2n}\cong SU(2n)/SO(2n)$. Take as base point
the Lagrangian space $L_o=\mathrm{Span}\{e_1,\ldots,e_{2n}\}$ of $\mathbb{R}^{4n}$ and let $\tau_0$ be the corresponding conjugation, so that the Cartan embedding of $\mathcal{L}^s_{2n}$ into $SU(2n)$ is given by $\tau=g\tau_og^{-1}\mapsto g\tau_0 g^{-1} \tau\in P_e^{\sigma_{\varrho,1}}$.

\begin{lem}\label{lemafixe} For each $\zeta\in \mathfrak{I}(SU(2n)^{\sigma_{\varrho,1}})$ we have $\exp\pi\zeta\in  P_e^{\sigma_{\varrho,1}}$.
\end{lem}
\begin{proof}
 Each $\zeta\in \mathfrak{I}(SU(2n)^{\sigma_{\varrho,1}})$ can be written as
 $\zeta=\sum_{i=1}^{n}n_i(E_i-E_{2n+1-i}).$
Hence, $\exp\pi\zeta=\pi_V-\pi_{V}^\perp$, where $V=\bigoplus_{{n_i\, \mathrm{even}}}\mathrm{Span}\{e_i,e_{2n+1-i}\}$.
 Define $g\in SU(2n)$ as follows: if $n_i$ is even, then $g(e_i)=e_i$ and $g(e_{2n+1-i})=e_{2n+1-i}$; if $n_i$ is odd, then  $g(e_i)=\mathrm{i} e_{i}$ and $g(e_{2n+1-i})=-\mathrm{i}e_{2n+1-i}$. We have $\exp\pi\zeta=g\tau_0g^{-1}\tau_0$, that is $\exp\pi\zeta\in P_e^{\sigma_{\varrho,1}}$.
\end{proof}

Now, identify $\mathcal{L}^s_{2n}$ with  $P_e^{\sigma_{\varrho,1}}$ via its Cartan embedding.
 By Theorem \ref{tinva}, any harmonic map $\varphi:S^2\to P_e^{\sigma_{\varrho,1}}$ admits an extended solution
$\Phi:S^2\setminus D\to U_{\zeta'}^{\sigma_{\varrho,1}}(SU(2n))$, for some $\zeta'\in \mathfrak{I}'(SU(2n))\cap\mathfrak{k}_{\sigma_{\varrho,1}}$ and some discrete subset $D$.  We can assume that  $\zeta'$ is a  $\varrho$-semi-canonical element. The corresponding $S^1$-invariant solution $u_\zeta\circ \Phi$ takes values in
 $\Omega_{\xi}(SU(2n)^{\sigma_{\varrho,1}})$, with $\xi\in \mathfrak{I}'(SU(2n)^{\sigma_{\varrho,1}})$; and both $\Phi_{-1}$ and  $(u_\zeta\circ \Phi)_{-1}$ take values in $P_\xi^{\sigma_{\varrho,1}}$.
A priori,  $\xi$ can be different from $\zeta$ since $\sigma_{\varrho,1}$ is not a fundamental outer involution. However, by Lemma \ref{lemafixe} we have $P_\xi^{\sigma_{\varrho,1}}=P_e^{\sigma_{\varrho,1}}=P_{\zeta'}^{\sigma_{\varrho,1}}.$

 If $\zeta$ is a $\varrho$-canonical element such that $\zeta'\preceq \zeta$ and $\mathcal{U}_{\zeta',\zeta'-\zeta}(\Phi)$ is constant, then,
    taking into account Proposition \ref{norminf}, there exists  a $T_\tau$-invariant extended solution $\tilde\Phi: S^2\setminus D\to U^\tau_{\zeta}(SU(2n)),$  where \begin{equation}\label{invo1}\tau=\mathrm{Ad}(\exp \pi(\zeta'-\zeta))\circ \sigma_{\varrho,1}.\end{equation} such that $\tilde\Phi_{-1}$ take values in $P_\zeta^\tau$ and $\varphi$ is given up to isometry
by
\begin{equation}\label{varphil}\varphi=\exp(\zeta'-\zeta)\tilde{\Phi}_{-1}\tau_0.\end{equation}
We conclude that, given a pair $(\zeta,\tau)$, where  $\zeta\in \mathfrak{I}(SU(2n)^{\sigma_\varrho})$ is a $\varrho$-canonical element and $\tau$ is an outer involution of the form (\ref{invo1}), any extended solution $\tilde{\Phi}:S^2\setminus D\to U_\zeta^\tau(SU(2n)))$ gives rise via (\ref{varphil}) to an harmonic map $\varphi$ from the two-sphere into $\mathcal{L}^s_{2n}$ and, conversely, all harmonic two-spheres into $\mathcal{L}^s_{2n}$ arise in this way.

\vspace{.10in}
For  $\mathcal{L}^s_{4}$,
 since $\exp\pi2\zeta_1$ belongs to the centre of $SU(4)$, we have five classes of harmonic maps into $\mathcal{L}^s_{4}$:
 $$(2\zeta_1,\sigma_{\varrho,1}),\, (\zeta_2,\sigma_{\varrho,1}),\,(2\zeta_1+\zeta_2,\sigma_{\varrho,1})\, (\zeta_2,\mathrm{Ad}\exp\pi\zeta_2\circ \sigma_{\varrho,1}),\,(2\zeta_1+\zeta_2,\mathrm{Ad}\exp\pi\zeta_2\circ\sigma_{\varrho,1}).$$

Let us consider in detail the class $(\zeta_2, \sigma_{\varrho,1})$. Clearly $r(\zeta_2)=2$.
Let $W_1$, $W_2$, $W_3$ and $W_4$ be the complex one-dimensional images of $E_1$, $E_2$, $E_3$ and $E_4$, respectively.
 That is, $W_i=\mathrm{Span}\{v_i\}$, where $v_i$ are defined by (\ref{ves}).
 Any extended solution $\Phi:S^2\setminus D\to U_{\zeta_2}^{\sigma_{\varrho,1}}$ is given by $\Phi=\exp C\cdot \gamma_{\zeta_2}$, with $\gamma_{\zeta_2}(\lambda)=\lambda^{-1}\pi_{W_4}+\pi_{W_3\oplus W_2 }+\lambda \pi_{W_1}$, for some holomorphic vector-valued function
$C:S^2\setminus D\to (\mathfrak{u}^0_{\zeta_2})_{\sigma_{\varrho,1}}$, where
$$(\mathfrak{u}^0_{\zeta_2})_{\sigma_{\varrho,1}}=(\mathfrak{p}_0^{\zeta_2})^\perp\cap\mathfrak{k}^\C_{\sigma_{\varrho,1}}+\lambda(\mathfrak{p}_1^{\zeta_2})^\perp\cap
\mathfrak{m}^\C_{\sigma_{\varrho,1}} $$
and
\begin{align*}
 (\mathfrak{p}_0^{\zeta_2})^\perp\cap\mathfrak{k}^\C_{\sigma_{\varrho,1}}&=(\g_{L_1-L_2}\oplus\g_{L_3-L_4} \oplus \g_{L_1-L_3}  \oplus \g_{L_2-L_4})\cap\mathfrak{k}^\C_{\sigma_{\varrho,1}},   \\ (\mathfrak{p}_1^{\zeta_2})^\perp\cap\mathfrak{m}^\C_{\sigma_{\varrho,1}}&= \g_{L_1-L_4}\cap\mathfrak{m}^\C_{\sigma_{\varrho,1}}= \g_{L_1-L_4}.
\end{align*}

 We have
$\sigma_{\varrho,1}(X_{1,2})=-X_{3,4}$ and $\sigma_{\varrho,1}(X_{1,3})= X_{2,4}$.    Hence we can write $C=C_0+C_1\lambda$, with
$$C_0=a(X_{1,2}-X_{3,4})+b(X_{1,3}+X_{2,4}),\,\,\,\,\,  C_1=cX_{1,4}$$
for some meromorphic functions $a,b,c$ on $S^2$. The harmonicity equations  impose that
$ab'-ba'=0$, which means that $b=\alpha a$ for some constant $\alpha\in \C.$
Hence given arbitrary  meromorphic functions $a,c$ on $S^2$ and a complex constant $\alpha$,
$$\Phi=\left(
                                          \begin{array}{cccc}
                                            1 & a & \alpha a & c\lambda \\
                                            0 & 1 & 0 & -\alpha a \\
                                            0 & 0 & 1 & a \\ 0 & 0 & 0 & 1  \\
                                          \end{array}
                                        \right)\cdot \gamma_{\zeta_1},$$
is an extended solution associated to some harmonic map in the class $(\zeta_2, \sigma_{\varrho,1})$. Reciprocally, any harmonic map in such class arises in this way.

\end{document}